\documentclass[10pt,twoside]{article}

\usepackage{latexsym}
\usepackage{amssymb,amsbsy,amsmath,amsfonts,amssymb,amscd}
\usepackage{hyperref}
\usepackage{pstricks}
\usepackage[latin1]{inputenc}
\usepackage{listings}
\usepackage{cancel}
\usepackage[pdftex]{graphicx, epsfig}
\usepackage{textcomp}
\usepackage{colortbl}
\usepackage{color}
\usepackage{stmaryrd}
\usepackage{fancyhdr}
\usepackage{multirow, array} 
\usepackage{tabulary}
\usepackage[rflt]{floatflt}

\newcommand{\commentout}[1]{}
\newcommand{\R}{\mathbb{R}}

\newcommand{\dis}{\displaystyle}

\newtheorem{theorem}{Theorem}

\newtheorem{lemma}[theorem]{Lemma}

\newtheorem{proposition}[theorem]{Proposition}
\newtheorem{remark}[theorem]{Remark}

\newcommand{\qed}{{ \hfill
                       {\unskip\kern 6pt\penalty 500 \raise -2pt\hbox{\vrule\vbox to 6pt{\hrule width 6pt
                       \vfill\hrule}\vrule} \par}   }}

\title{Long time evolutionary dynamics of phenotypically structured populations in time-periodic environments}

\author{
Susely Figueroa Iglesias\thanks{Institut de Math\'ematiques de Toulouse; UMR 5219, Universit\'e de Toulouse; CNRS, UPS IMT, F-31062 Toulouse Cedex 9, France; E-mail: Susely.Figueroa@math.univ-toulouse.fr}, \and Sepideh Mirrahimi\thanks{Institut de Math\'ematiques de Toulouse; UMR 5219, Universit\'e de Toulouse; CNRS, UPS IMT, F-31062 Toulouse Cedex 9, France; E-mail: Sepideh.Mirrahimi@math.univ-toulouse.fr} }
\date{\today}

\begin{document}
\maketitle
\begin{abstract}
We study the long time behavior of a parabolic Lotka-Volterra type equation considering a time-periodic growth rate with non-local competition. Such equation describes the dynamics of a phenotypically structured population under the effect of mutations and selection in a fluctuating environment. We first prove that, in long time, the solution converges to the unique periodic solution of the problem. Next, we des\-cribe this periodic solution asymptotically as the effect of the mutations vanish. Using a theory based on Hamilton-Jacobi equations with constraint, we prove that, as the effect of the mutations vanishes, the solution concentrates on a single Dirac mass, while the size of the population varies periodically in time. When the effect of the mutations is small but nonzero, we provide some formal approximations of the moments of the population's distribution. We then show, via some examples, how such results could be compared to biological experiments.
\end{abstract}
\pagenumbering{arabic}

\section{Introduction}
\subsection{Model and motivations}
The purpose of this article is to study the evolutionary dynamics of a phenotypically structured population in a time-periodic environment. While the evolutionary dynamics of populations in constant environments are widely studied (see for instance \cite{DJMP, Champagnat, DJMR, Raoul-Phd, Mirraihimi-Phd}), the theoretical results on varying environments remain limited (see however \cite{Lorenzi-Desvillettes, Souganidis}).
The variation of the environment may for instance come from the seasonal effects or a time varying administration of medications to kill cancer cells or bacteria.  Several questions arise related to the time fluctuations. Could a population survive under the fluctuating change? How the population size will be affected? Which phenotypical trait will be selected? What will be the impact of the variations of the environment on the population's phenotypical distribution?
\\

Several frameworks have been used to study the dynamics of populations under selection and mutations. Game theory has been one of the first approaches to study evolutionary dynamics \cite{Hofbauer-Sigmund, smith1982evolution}. Adaptive dynamics which is a theory based on the stability of dynamical systems allows to study evolution under rare mutations \cite{Dieckmann-Law, DiekmannBook}. Integro-differential models are used to study evolutionary dynamics of large populations (see for instance \cite{Calsina-Cuadrado, DJMR, DJMP, Magal-Webb}). Probabilistic tools allow to study populations of small size \cite{Champagnat-Lambert} and also to derive the above models in the limit of large populations \cite{Champagnat}.\\
Here, we are interested in the integro-differential approach. We study in particular the following Lotka-Volterra type model
\begin{equation}
\label{CompleteModel}
\left\{
\begin{array}{l}
\partial_t n(t,x)-\sigma\Delta  n(t,x)=n(t,x)[a(t,x)-\rho(t)],\quad(t,x)\in[0,+\infty)\times\R^d,\\
\rho(t)=\dis\int_{\R^d}n(t,x)dx,\\
n(t=0,x)=n_0(x).
\end{array}
\right.
\end{equation}
Here, $n(t,x)$ represents the density of individuals with trait $x$ at time $t$. The mutations are represented by a Laplace term with rate $\sigma$. The term $a(t,x)$ is a time-periodic function, corresponding to the net growth rate of individuals with trait $x$ at time $t$. We also consider a death term due to  competition between the individuals, whatever their traits, proportional to the total population size $\rho(t)$.\\

A main part of our work is based on an approach using Hamilton-Jacobi equations with constraint. This approach has been developed during the last decade to study asymptotically the   dynamics of populations under selection and small mutations. There is a large literature on this approach. We refer for instance to \cite{Mirraihimi-Phd, Perthame_Barles} where the basis of this approach for problems coming from evolutionary biology were established.  Note that related tools were already used to study  the propagation phenomena for local reaction-diffusion equations \cite{Evans&Souganidis, Freidlin}.\\

Our work follows an earlier article on the analysis of phenotype-structured  populations in time-varying environments \cite{Souganidis}. In \cite{Souganidis}, the authors study a similar equation to \eqref{CompleteModel} using also the Hamilton-Jacobi approach, but with a different scaling than our paper. They indeed obtain a homogenization result by  simultaneously accelerating time and letting the size of the mutations vanish. In this paper, we study first a long time limit of this equation and next we describe asymptotically such long time solutions as the effect of the mutations vanishes. Our scaling, being motivated by biological applications (see Section \ref{examples_bio}), leads to a different qualitative behavior of solutions and requires a totally different mathematical analysis.

\subsection{Assumptions}
To introduce our assumptions, we first define 
$$
\overline{a}(x)=\dis\frac{1}{T}\int_0^Ta(t,x)dt.
$$
We then assume that $a(t,x)$ is a time-periodic function with period $T$, and $C^3$ with respect to $x$, such that
\begin{equation}\tag{H1}
\label{a_W3inf}
a(t,x)=a(t+T,x),\ \forall\ (t,x)\in\R\times\R^d,\quad\text{and}\ \exists\ d_0>0: \Vert a(t,\cdot)\Vert_{L^{\infty}(\R^d)}\leq d_0\quad\forall\ t\in\R.
\end{equation}
Moreover, we suppose that there exists a unique $x_m$  which satisfies for some constant $a_m$,
\begin{equation}\tag{H2}
\label{x_m}
0< a_m \leq\max_{x\in\R^d} \overline{a}(x)=\overline{a}(x_m).
\end{equation}
In order to guarantee that the initial condition do not explode, we make the following assumption
\begin{equation}\tag{H3}
\label{n0_exp}
0\leq n_0(x)\leq e^{C_1-C_2|x|}, \quad \forall x\in\R^d,
\end{equation}
for some positive constants $C_1$, $C_2$.\\
Furthermore, just for Section \ref{NoMutations}, that is the case with $\sigma=0$, we assume additionally that
\begin{equation}\tag{H4}
\label{Hessian}
H=
\left(\frac{\partial^2 \overline{a}}{\partial x_i\partial x_j}(x_m)\right)_{i,j}\quad \text{is negative definite,}
\end{equation}
i.e., its eigenvalues are all negative. Also, let us suppose that there exist some positive constants $\delta$ and $R_0$ such that
\begin{equation}\tag{H5}
\label{a_neg}
a(t,x)\leq-\delta, \quad \text{for all } t\geq0 , \text{ and }|x|\geq R_0.
\end{equation}
Finally, let $M$ and $d_1$ be positive constants, it is assumed again for the case of no mutations, that 
\begin{equation}\tag{H6}
\label{n0_W3inf}
\Vert n_0\Vert_{W^{3,\infty}}\leq M,\qquad \Vert a\Vert_{W^{3,\infty}}\leq d_1.
\end{equation}

\subsection{Main results}

We begin the qualitative study, with a simpler case, where $\sigma=0$, which means there is no mutation. The model reads as follows
\begin{equation}
\label{model_sigma0}
\left\{
\begin{array}{l}
\partial_t n(t,x)=n(t,x)[a(t,x)-\rho(t)],\quad(t,x)\in[0,+\infty)\times\R^d,\\
\rho(t)=\dis\int_{\R^d}n(t,x)dx,\\
n(t=0,x)=n_0(x).
\end{array}
\right.
\end{equation}
Our first result is the following.
\begin{proposition}{(case $\sigma=0$)}\\
\label{PrinTheo}
Assume \eqref{a_W3inf}-\eqref{n0_W3inf}. Let $n$ be the solution of \eqref{model_sigma0}. Then,
\begin{itemize}
\item[(i)] as $t\rightarrow+\infty$, $\Vert\rho(t)-\widetilde{\varrho}(t)\Vert_{L^\infty}\rightarrow0$, where $\widetilde{\varrho}(t)$ is the unique positive periodic solution of equation
\begin{equation}
\label{SysRhoPer}
\left\{\begin{array}{l}
\dfrac{d\widetilde{\varrho}}{dt}=\widetilde{\varrho}(t)\left(a(t,x_m)-\widetilde{\varrho}(t)\right),\quad t\in(0,T),\\
\widetilde{\varrho}(0)=\widetilde{\varrho}(T),
\end{array}\right.
\end{equation}
given by
\begin{equation}
\label{Exp_rho}
\widetilde{\varrho}(t)=\frac{1-\exp\left[-\dis\int_0^{T}a(s,x_m)ds\right]}{\exp\left[-\dis\int_0^{T}a(s,x_m)ds\right]\dis\int_t^{t+T}\exp\left[\dis\int_{t}^s a(\theta,x_m)d\theta\right]ds}.
\end{equation} 
\item[(ii)] Moreover, $\dfrac{n(t,x)}{\rho(t)}$ converges weakly in the sense of measures to $\delta(x-x_m)$ as $t\rightarrow+\infty$. As a consequence,
$$ 
n(t,x)-\widetilde{\varrho}(t)\delta(x-x_m)\rightharpoonup 0\qquad\mathrm{as}\;t\rightarrow+\infty,
$$
in the sense of measures.
\end{itemize}
\end{proposition}
This result implies that the trait with the highest time average of the net growth rate over the time interval $[0, T]$, will be selected in long time, while the size of the population oscillates with environmental fluctuations.\\ 

\noindent
To present our results for problem \eqref{CompleteModel}, we first introduce the following parabolic eigenvalue problems 
\begin{equation}
\left\{
\label{EigenPb}
\begin{array}{l}
\partial_t p-\sigma\Delta p-a(t,x)p=\lambda p,\quad\mathrm{in}\;[0,+\infty)\times\R^d,\\
0<p:\;T-\mathrm{periodic},
\end{array}
\right.
\end{equation}
\begin{equation}
\label{EqLamdaR}
\left\{
\begin{array}{cr}
\partial_t p_R-\sigma\Delta p_R-a(t,x)p_R=\lambda_R p_R,&\mathrm{in}\;[0,+\infty)\times B_R,\\
p_R=0,&\mathrm{on}\;[0,+\infty)\times\partial B_R,\\
0<p_R:\ T-\mathrm{periodic},
\end{array}
\right.
\end{equation}
where $B_R$ is the ball in $\R^d$ centered at the origin with radius $R>0$. It is known that (see \cite{hess}) if $a\in L^\infty([0,+\infty)\times B_R)$, then there exists a unique principal eigenpair $(\lambda_R,p_R)$  for \eqref{EqLamdaR} with $\Vert p_R(0,\cdot)\Vert_{L^\infty(B_R)}=1$. Moreover, as $R\to +\infty$, $\lambda_R\searrow \lambda$ and $p_R$ converges along subsequences to $p$, with $(\lambda,p)$ solution of \eqref{EigenPb} (see for instance  \cite{huska08}).
\\
We next assume a variant of hypothesis \eqref{a_neg}, 
that is, there exist positive constants, $\delta$ and $R_0$  such that
\begin{equation}\tag{$H5_\sigma$}
\label{a_lambda_neg}  
a(t,x)+\lambda\leq-\delta, \quad \text{for all } 0\leq t, \text{ and }R_0\leq |x|.
\end{equation}
Under the above additional assumption, which means that ``$a$" takes small values at infinity, the eigenpair $(\lambda,p)$ is also unique, (see Lemma \ref{existence_eigen}).\\
We next define the $T-$periodic functions $Q(t)$ and $P(t,x)$ as follows
\begin{equation}
\label{PyQ}
Q(t)=\dfrac{\int_{\R^d}a(t,x)p(t,x)dx}{\int_{\R^d}p(t,x)dx},\quad P(t,x)=\dfrac{p(t,x)}{\int_{\R^d}p(t,x)dx}.
\end{equation}
We deduce from previous Proposition that if and only if $\int_0^TQ(t)>0$, then there exists a unique positive periodic solution $\widetilde{\rho}(t)$ for the problem 
$$
\left\{\begin{array}{l}
\dfrac{d\widetilde{\rho}}{dt}=\widetilde{\rho}\left[Q(t)-\widetilde{\rho}\right],\quad t\in(0,T),\\
\widetilde{\rho}(0)=\widetilde{\rho}(T).
\end{array}\right.
$$
We can then describe the long time behavior of the solution of \eqref{CompleteModel}
\begin{proposition}{(case $\sigma>0$, long time behavior)}\\
\label{PrinTheoSect2}
Assume \eqref{a_W3inf}, \eqref{x_m}, \eqref{n0_exp} and \eqref{a_lambda_neg}. Let $n$ be the solution of \eqref{CompleteModel}, then 
\begin{itemize}
\item[(i)] if $\lambda\geq 0$ then the population will go extinct, i.e.
$\rho(t)\rightarrow0,$ as $t\rightarrow\infty$, 
\item[(ii)] if $\lambda<0$ then $\vert\rho(t)-\widetilde{\rho}(t)\vert\rightarrow0$, as $t\rightarrow\infty.$
\item[(iii)] Moreover $\quad\left\Vert\dfrac{n(t,x)}{\rho(t)}-P(t,x)\right\Vert_{L^\infty}\longrightarrow0$, as $t\rightarrow\infty$. Consequently we have, as $t\rightarrow\infty$
\begin{equation}
\label{convPerio}
\Vert n(t,\cdot)-\widetilde{\rho}(t)P(t,\cdot)\Vert_{L^\infty}\rightarrow 0,\:\mathrm{if}\:\lambda<0\quad\mathrm{and}\quad\Vert n\Vert_{L^\infty}\rightarrow0,\:\mathrm{if}\:\lambda\geq0.
\end{equation}
\end{itemize} 
\end{proposition}
\begin{remark}
\label{Rk_in_a}
Assuming \eqref{x_m} implies that $\lambda<0$, provided $\sigma$ is small enough. 
\end{remark}
We prove this remark in Lemma \ref{lam_eps_neg}.\\

\noindent
Proposition \ref{PrinTheoSect2} guarantees, when $\lambda<0$, the convergence in $L^\infty-$norm of the solution $n(t,x)$ of the equation \eqref{CompleteModel} to the periodic function $\widetilde{n}(t,x)=\widetilde{\rho}(t)P(t,x)$ and it is not difficult to verify that $\widetilde{n}$ is in fact a solution of \eqref{CompleteModel}.\\ 
We next describe the periodic solution  $\widetilde{n}$, asymptotically as the effect of mutations is small. To this end, with a change of notation, we take $\sigma=\varepsilon^2$ and study $(n_\varepsilon,\rho_\varepsilon)$, the unique periodic solution of the following equation
\begin{equation}
\label{Pb_n_Epsilon1}
\left\{
\begin{array}{lr}
\partial_t n_\varepsilon-\varepsilon^2\Delta n_\varepsilon=n_\varepsilon[a(t,x)-\rho_\varepsilon(t)],& (t,x)\in[0,+\infty)\times\R^d,\\
\rho_\varepsilon(t)=\dis\int_{\R^d}n_\varepsilon(t,x)dx,\\
n_\varepsilon(0,x)=n_\varepsilon(T,x).
\end{array}
\right.
\end{equation}
We expect that $n_\varepsilon(t, x)$ concentrates as a Dirac mass as $\varepsilon\rightarrow0$.\\
In order to study the limit of $n_\varepsilon$, as $\varepsilon\rightarrow0$,   we make the Hopf-Cole transformation
\begin{equation}
\label{HopfCole}
n_\varepsilon=\frac{1}{(2\pi\varepsilon)^{d/2}}\exp{\left(\frac{u_\varepsilon}{\varepsilon}\right)},
\end{equation} 
which allows us to prove
\begin{theorem}{(case $\sigma=\varepsilon^2$, asymptotic behavior)}\\
\label{Prin_Theo_ueps}
Let $n_\varepsilon$ solve \eqref{Pb_n_Epsilon1} and assume \eqref{a_W3inf}, \eqref{x_m} and \eqref{a_lambda_neg}. Then 
\begin{itemize}
\item[(i)] As $\varepsilon\rightarrow0$, we have
\begin{equation}
\label{ConvTheo_eps}
\Vert\rho_\varepsilon(t)-\widetilde{\varrho}(t)\Vert_{L^\infty}\rightarrow0,\quad\mathrm{and}\quad n_\varepsilon(t,x)-\widetilde{\varrho}(t)\delta(x-x_m)\rightharpoonup 0,
\end{equation}
point wise in time, weakly in $x$ in the sense of measures, with $\widetilde{\varrho}(t)$ given by \eqref{Exp_rho}.
\item[(ii)] Moreover as $\varepsilon\rightarrow0$, $u_\varepsilon$ converges locally uniformly to a function $u(x)\in C(\R)$, the unique viscosity solution to the following equation
\begin{equation}
\label{LimitEq}
\left\{\begin{array}{l}
-|\nabla u|^2=\dfrac{1}{T}\dis\int_0^T(a(t,x)-\widetilde{\varrho}(t))dt,\quad x\in\R^d,\\
\dis\max_{x\in\R^d}u(x)=u(x_m)=0.
\end{array}
\right.
\end{equation}
In the case $x\in\R$, $u$ is indeed a classical solution and is given by
\begin{equation}
\label{Exp_Sol_u}
u(x)=-\left\vert\int_{x_m}^x\sqrt{-\overline{a}(x')+\overline{\varrho}}\;dx'\right\vert
\end{equation} 
where $\overline{\varrho}=\frac{1}{T}\int_0^T\widetilde{\varrho}(t)dt$.
\end{itemize}
\end{theorem} 

\noindent
To prove Theorem \ref{Prin_Theo_ueps}, we first prove some regularity estimates on $u_\varepsilon$ and then  pass to the limit in the viscosity sense using the method of perturbed test functions. We finally show that \eqref{LimitEq} has a unique solution, and hence all the sequence converges. Note that in order to prove regularity estimates on $u_\varepsilon$, a difficulty comes from the fact that $u_\varepsilon$ is time-periodic and one cannot use, similarly to previous related works \cite{MAA, lorz}, the bounds on the initial condition to obtain such bounds for all time and further work is required.

\subsection{Some heuristics and the plan of the paper }

We next provide some heuristic computations which allow to better understand Theorem \ref{Prin_Theo_ueps}, but also suggest an approximation of the population's distribution $n_\varepsilon$, when $\varepsilon$ is small but nonzero.\\
Replacing \eqref{HopfCole} in \eqref{Pb_n_Epsilon1}, we first notice that $u_\varepsilon$ solves
\begin{equation}
\label{Pb_u_Epsilon}
\left\{
\begin{array}{rllr}
\frac{1}{\varepsilon}\partial_t u_\varepsilon-\varepsilon\Delta u_\varepsilon&=&|\nabla u_\varepsilon|^2+a(t,x)-\rho_\varepsilon(t),&(t,x)\in[0,+\infty)\times\R^d,\\
u_\varepsilon(t=0,x)&=&u_\varepsilon^0(x)=\varepsilon\ln n_\varepsilon^0(x).
\end{array}
\right.
\end{equation}
We then write formally an asymptotic expansion for $u_\varepsilon$ and $\rho_\varepsilon$ in powers of $\varepsilon$
\begin{equation}
\label{App_u}
u_\varepsilon(t,x)=u(t,x)+\varepsilon v(t,x)+\varepsilon^2 w(t,x)+o(\varepsilon^2),\quad\rho_\varepsilon(t)=\rho(t)+\varepsilon\kappa(t)+o(\varepsilon),
\end{equation}
where the coefficients of the developments are time-periodic.\\
We substitute in \eqref{Pb_u_Epsilon} and organize by powers of $\varepsilon$, that is
$$
\frac{1}{\varepsilon}\left(\partial_t u(t,x)\right)+\varepsilon^0\left[\partial_t v(t,x)-|\nabla u|^2-a(t,x)+\rho(t)\right]+\varepsilon\left[\partial_t w-\Delta u-2\nabla u\cdot \nabla v+\kappa(t)\right]+o(\varepsilon^2)=0.
$$
From here we obtain
$$
\partial_t u(t,x)=0\Leftrightarrow u(x,t)=u(x),
$$
and
$$
\partial_t v(t,x)=|\nabla u|^2+a(t,x)-\rho(t).
$$
Integrating this latter equation in $t\in[0,T]$, we obtain that
$$
0=\int_0^T|\nabla u|^2dt+\int_0^Ta(t,x)dt-\int_0^T\rho(t)dt,
$$
because of the $T-$periodicity of $v$. This implies that
$$
-|\nabla u|^2=\frac{1}{T}\int_0^T(a(t,x)-\rho(t))dt,
$$
which is the first equation in \eqref{LimitEq}.
Keeping next the terms of order $\varepsilon$ we obtain that
$$
\partial_t w-\Delta u=2\nabla u\cdot \nabla v-\kappa(t),
$$
and again integrating in $[0,T]$ we find
$$
-\Delta u=\frac{2}{T}\nabla u\int_0^T\nabla vdt-\overline{\kappa},\qquad\mathrm{with}\quad\overline{\kappa}=\dis\frac{1}{T}\int_0^T\kappa(t)dt.
$$
Evaluating the above equation at $x_m$ we obtain that 
$$
\Delta u(x_m)=\overline{\kappa}.
$$
Then, using the averaged coefficients 
$\overline{a}(x)=\frac{1}{T}\int_0^Ta(t,x)dt$ and $\overline{\rho}=\frac{1}{T}\int_0^T\rho(t)dt$, 
we deduce, combining the above computations, that  $v(t,x)$ satisfies
\begin{equation}
\label{Eq_v_Lap_u}
\left\{
\begin{array}{rcl}
\partial_t v&=&a(t,x)-\overline{a}(x)-\rho(t)+\overline{\rho},\\
-\Delta u&=&\dfrac{2}{T}\dis\int_0^T\nabla u\cdot\nabla v dt-\overline{\kappa},
\end{array}
\right.
\end{equation}
which allows to determine $v$.\\
We will use these formal expansions in Section \ref{Moments}, to estimate the moments of the population's distribution using the Laplace's method of integration. Note that such approximations were already used to study the phenotypical distribution of a population in a spatially heterogeneous environment \cite{Mirrahimi17, Mirrahimi&Gandon} (see also \cite{Mirrahimi-Roquejoffre-Cras} where such type of approximation was first suggested).  We next show, via two examples, how such results could be interpreted biologically. In particular, our work being motivated by a biological experiment in  \cite{Ketola}, we suggest a possible explanation for a phenomenon observed in this experiment.

\bigskip

\noindent
The paper is organized as follows. In Section \ref{NoMutations} we deal with problem \eqref{model_sigma0}  and prove Proposition \ref{PrinTheo}. In Section \ref{SigmaPositive} we study the long time behavior of  \eqref{CompleteModel} and provide the proof of Proposition \ref{PrinTheoSect2}. Next, in Section \ref{SmallMutations} we study the asymptotic behavior of $n_\varepsilon$ as $\varepsilon\to 0$, and  prove Theorem \ref{Prin_Theo_ueps}. In Section \ref{Moments}, we use the above formal arguments to estimate the moments of the population's phenotypical distribution. Finally we use these results in Section \ref{examples_bio} to study two biological examples considering two different growth rates. 

\section{The case with no mutations}
\label{NoMutations}
In this section we study the qualitative behavior of \eqref{model_sigma0}, where $\sigma=0$, and provide the proof of Proposition \ref{PrinTheo}.\\
To this end, we define $N(t,x)=n(t,x)e^{\int_0^t\rho(s)ds}$ which solves
$$
\partial_t N=a(t,x)N(t,x).
$$
From the periodicity of $a$ and the Floquet theory we obtain that $N$ has the following form
$$
N(t,x)=e^{\mu(x)t}p_0(t,x),\quad\mathrm{with}\quad p_0(0,x)=p_0(T,x),\quad\mathrm{and}\quad \mu(x)=\overline{a}(x)=\frac{1}{T}\int_0^Ta(s,x)ds.
$$

\subsection{Long time behavior of $\rho$}
In this subsection we prove Proposition \ref{PrinTheo}  $(i)$.\\ Integrating equation \eqref{model_sigma0} with respect to $x$ we obtain
\begin{equation}
\label{Eq_Rho}
\frac{d}{dt} \rho(t)=\int_{\R^d}n(t,x)a(t,x)dx-\rho(t)^2=\rho(t)\left[\int_{\R^d}\frac{n(t,x)a(t,x)}{\rho(t)}dx-\rho(t)\right].
\end{equation} 
Then we claim the following Lemma that we prove at the end of this subsection.
\begin{lemma}
\label{TechPropSect2}
Assume \eqref{a_W3inf}-\eqref{n0_exp} and \eqref{n0_W3inf} then 
$$
\left\vert\int_{\R^d}\frac{n(t,x)a(t,x)}{\rho(t)}dx -a(t,x_m)\right\vert\longrightarrow 0,\quad\mathrm{as}\;t\rightarrow+\infty.
$$
\end{lemma}
\begin{proof} (Proposition \ref{PrinTheo})(i)\\
From Lemma \ref{TechPropSect2},  $\rho(t)$ satisfies
$$
\frac{d}{dt}\rho(t)=\rho(t)[a(t,x_m)+\Sigma(t)-\rho(t)],
$$
where 
$\Sigma(t)\rightarrow 0$ as $t\rightarrow\infty$.
In order to prove the convergence to a periodic function, we adapt a method introduced in \cite{Lopez}.\\
After a standard substitution $\kappa(t)=1/\rho(t)$ in order to linearize the latter equation, and integration with the help of an integrating factor, the solution $\rho$ can be written as follows
$$
\frac{1}{\rho(t)}=\exp\left(-\int_0^t(a(s,x_m)+\Sigma(s))ds\right)\left(\frac{1}{\rho_0}+\int_0^t\exp\left(\int_0^s(a(\theta,x_m)+\Sigma(\theta))d\theta\right)ds\right).
$$
We then write $\dfrac{1}{\rho((k+1)T)}$ as function of $\dfrac{1}{\rho(kT)}$, that is
$$
\dfrac{1}{\rho((k+1)T)}=\exp\left(-\int_{kT}^{(k+1)T}(a(s,x_m)+\Sigma(s))ds\right)\left(\dfrac{1}{\rho(kT)}+\dis\int_{kT}^{(k+1)T}e^{\int_{kT}^s(a(\theta,x_m)+\Sigma(\theta))d\theta}ds\right),
$$
and we obtain a recurrent sequence for $\rho_k=\rho(kT)$ as follows
$$
\dfrac{1}{\rho_{k+1}}=\xi_k+\dfrac{\eta_k}{\rho_k},
$$
where
$$
\eta_k=\exp\left(-\int_{kT}^{(k+1)T}(a(s,x_m)+\Sigma(s))ds\right),\quad\xi_k=\eta_k\dis\int_{kT}^{(k+1)T}\exp\left(\int_{kT}^s(a(\theta,x_m)+\Sigma(\theta))d\theta\right)ds.
$$
From the $T-$periodicity of $a$ and the fact that $\Sigma(t)\rightarrow0$ we obtain easily that $\eta_k\rightarrow \eta$ and $\xi_k\rightarrow \xi$ as $k\rightarrow\infty$, where
$$
\eta=\exp\left(-\int_0^Ta(t,x_m)dt\right),\quad\xi=\eta\int_0^T\exp\left(\int_0^ta(\theta,x_m)d\theta\right)dt.
$$
From these convergences we have that for all $\epsilon>0$, there exists $K_\epsilon$ such that
$$
\xi-\epsilon\leq \xi_k\leq \xi+\epsilon,\qquad\eta-\epsilon\leq \eta_k\leq \eta+\epsilon,\quad \forall\;k\geq K_\epsilon,
$$
which implies
$$
\xi-\epsilon+\frac{\eta-\epsilon}{\rho_k}\leq\frac{1}{\rho_{k+1}}\leq \xi+\epsilon+\frac{\eta+\epsilon}{\rho_k}.
$$
Note $\kappa_k=\dfrac{1}{\rho_k}$ then 
\begin{equation}
\label{Eq_kappa}
\xi-\epsilon+(\eta-\epsilon)\kappa_k\leq\kappa_{k+1}\leq \xi+\epsilon+(\eta+\epsilon)\kappa_k.
\end{equation}
From the  inequality  at the right hand side of \eqref{Eq_kappa}, denoting $\kappa^*=\dis\limsup_{k\rightarrow+\infty} \kappa_k$, we obtain
$$
\kappa^*\leq \xi+\epsilon+(\eta+\epsilon)\kappa^*,\qquad \forall\epsilon>0.
$$
Then thanks to assumption \eqref{x_m}, which implies $\eta<1$, we have
$$
\kappa^*\leq \frac{\xi}{1-\eta}.
$$
Analogously, from the left hand side inequality in \eqref{Eq_kappa}, and denoting $\kappa_*=\dis\liminf_{k\rightarrow+\infty} \kappa_k$, we deduce that  
$$
\kappa_*\geq \frac{\xi}{1-\eta}.
$$
Since $\kappa_*\leq\kappa^*$, we obtain
$$
\kappa^*=\kappa_*=\lim_{k\rightarrow+\infty}\kappa_k=\frac{\xi}{1-\eta}.
$$
Going back to variable $\rho_k$, it implies
$$
\lim_{k\rightarrow\infty} \rho_k=\dfrac{1-\eta}{\xi}.
$$
Finally we can make a translation from $\rho_0$ to obtain 
$$
\widetilde{\varrho}(t)=\lim_{k\rightarrow\infty}\rho(kT+t),
$$
with $\widetilde{\varrho}(t)$ the unique periodic solution of equation \eqref{SysRhoPer}
given by \eqref{Exp_rho}.
\begin{flushright}
$\blacksquare$
\end{flushright}
\end{proof}
Finally we prove Lemma \ref{TechPropSect2}.\\
\begin{proof}
Let $K=\{x\in\R^d:|x|\leq R_0\}$ for $R_0$ as in assumption \eqref{a_neg} then 
$$
\begin{array}{l}
\dis\int_{\R^d}\dfrac{n(t,x)a(t,x)}{\rho(t)}dx= \dfrac{\dis\int_{K^c}n(t,x)a(t,x)dx+\int_{K}p_0(t,x)e^{t\mu(x)}a(t,x)dx}{\dis\int_{K^c}n(t,x)dx+\int_{K}p_0(t,x)e^{t\mu(x)}dx}.
\end{array}
$$
Thanks to \eqref{model_sigma0} and assumptions \eqref{a_W3inf}, \eqref{n0_exp} and \eqref{a_neg}  we can control the integral terms taken outside the compact set $K$ as follows
\begin{equation}
\label{Bound_n_outCompact}
\dis\int_{K^c}n(t,x)a(t,x)dx\leq \Vert a\Vert_{L^\infty}e^{-\delta t}\dis\int_{K^c}n_0(x)dx\leq C e^{-\delta t} \int_{{K}^c}e^{-C_2|x|}dx\longrightarrow 0,\quad\mathrm{as}\: t\rightarrow\infty,
\end{equation}
and an analogous inequality holds for $\int_{K^c}n(t,x)dx$. Next for the remaining terms, we use Taylor expansions around the point $x=x_m$  until third order terms, for $x_m$ given by \eqref{x_m}, that is,
$$
\begin{array}{rcl}
I(t)&=&\dis\int_{K}p_0(t,x)e^{t\mu(x)}a(t,x)dx\\\\
&=&\dis\int_{K}\left[a(t,x_m)+\nabla a(t,x_m)(x-x_m)+\frac{1}{2}\,^t(x-x_m)D^2a(t,x_m)(x-x_m)+O(|x-x_m|^3)\right]\\
&&\quad\cdot\dis\left[p_0(t,x_m)+\nabla p_0(t,x_m)(x-x_m)+\frac{1}{2}\,^t(x-x_m)D^2a(t,x_m)(x-x_m)+O(|x-x_m|^3)\right]\\
&&\;\cdot\mathrm{exp}\left\{\dis\frac{t}{2}\,^t(x-x_m)D^2 \mu(x_m)(x-x_m)+tO(|x-x_m|^3)\right\}dx,
\end{array}
$$
where $^t x$ indicates the transpose vector of $x$.\\
We organize $I(t)$ by powers of $|x-x_m|$ as below
$$
\begin{array}{lll}
I_0(t)&=&a(t,x_m)p_0(t,x_m)\dis\int_{K} e^{\frac{t}{2} \,^t(x-x_m)D^2 \mu(x_m)(x-x_m)}dx,\\
I_1(t)&=&\dis\int_{K}\left[a(t,x_m)\nabla p_0(t,x_m)+p_0(t,x_m)\nabla a(t,x_m)\right] (x-x_m)e^{\frac{t}{2}\,^t(x-x_m)D^2 \mu(x_m)(x-x_m)}dx=0,\\
I_2(t)&=&\dis\int_{K}\left\{\ ^t(x-x_m)\left[\frac{1}{2}a(t,x_m)D^2p_0(t,x_m)+\,^t\nabla a(t,x_m)\nabla p_0(t,x_m)+\frac{1}{2}p_0(t,x_m)D^2 a(t,x_m)\right] (x-x_m)\right.\\
&&\qquad\left. \cdot\ e^{\frac{t}{2}\,^t(x-x_m)D^2 \mu(x_m)(x-x_m)}\right\}dx,\\
I_3(t)&=&\dis\int_{K} (1+t)O(|x-x_m|^3)e^{\frac{t}{2}\,^t(x-x_m)D^2 \mu(x_m)(x-x_m)}dx.
\end{array}
$$
By performing a change of variables as $y=\sqrt{t}(x-x_m)$ we obtain for the non null integrals
$$
\begin{array}{lll}
I_0(t)&=&\dis\frac{1}{t^{d/2}}a(t,x_m)p_0(t,x_m)\int_{\widetilde{K}_t} e^{\frac{1}{2} \,^tyD^2 \mu(x_m)y}dy,\\\\
I_2(t)&=&\dis\frac{1}{t^{d/2+1}}\int_{\widetilde{K}_t}\ ^ty\left[\frac{1}{2}a(t,x_m)D^2p_0(t,x_m)+\ ^t\nabla a(t,x_m)\nabla p_0(t,x_m)+\frac{1}{2}p_0(t,x_m)D^2 a(t,x_m)\right] ye^{\frac{1}{2}\,^tyD^2 \mu(x_m)y}dy,\\\\
I_3(t)&=&\dis\frac{1+t}{t^{\frac{d+3}{2}}}\int_{\widetilde{K}_t} O(|y|^3)e^{\frac{1}{2}\,^tyD^2 \mu(x_m)y}dy\approx O\left(\frac{1}{t^{\frac{d+1}{2}}}\right),
\end{array}
$$
with $\widetilde{K}_t=\{x\in\R^d:|x|\leq \sqrt{t}R_0\}$.\\
Note that $I_1(t)=0$ because it is the integral of an odd function in a symmetric interval. Moreover we obtain the approximation for $I_3(t)$ thanks to assumption \eqref{n0_W3inf}, which implies that the derivatives of $\mu$ and $a$, and consequently $p_0$, up to order 3, are globally bounded.\\
Moreover, if we denote $A(t)=\Big(\alpha_{ij}(t)\Big)_{i,j}$ the periodic matrix inside the crochets in $I_2(t)$, i.e
$$
A(t)=\frac{1}{2}a(t,x_m)D^2p_0(t,x_m)+\,^t\nabla a(t,x_m)\nabla p_0(t,x_m)+\frac{1}{2}p_0(t,x_m)D^2 a(t,x_m),
$$
we obtain, thanks to the periodicity of $a$ and $p_0$, that all the coefficients of $A(t)$ are bounded as $t\rightarrow\infty$.
Moreover,
$$
\left|\,^tyA(t)y\right|=\left|\sum_{i,j=1}^d\alpha_{ij}(t)y_iy_j\right|\leq \sum_{i,j=1}^d|\alpha_{ij}(t)||y_i||y_j|\leq C|y|^2,\ \text{for some}\ C>0.
$$
Then for $I_2$ we have, by using \eqref{Hessian}
$$
\left|I_2(t)\right|\leq \dis\frac{C}{t^{d/2+1}}\int_{\widetilde{K}_t} |y|^2e^{\frac{1}{2}\,^tyD^2 \mu(x_m)y}dy\approx O\left(\frac{1}{t^{d/2+1}}\right)\qquad\mathrm{as}\qquad t\rightarrow\infty.
$$
By arguing in the same way we obtain for the denominator term
$$
\int_{K}p_0(t,x)e^{t\mu(x)}dx= \frac{1}{t^{d/2}}p_0(t,x_m)\dis\int_{\widetilde{K}_t}e^{\frac{1}{2}\,^tyD^2 \mu(x_m)y}dy+O\left(\frac{1}{t^{\frac{d+1}{2}}}\right),
$$
and we conclude multiplying by $t^{d/2}$ and using again \eqref{Hessian}
$$
\dis\int_{\R^d}\dfrac{n(t,x)a(t,x)}{\rho(t)}dx=\dfrac{a(t,x_m)p_0(t,x_m)\dis\int_{\widetilde{K}_t}e^{\frac{y^2}{2}D^2\mu(x_m)}dy+O\left(\frac{1}{\sqrt{t}}\right)}{p_0(t,x_m)\dis\int_{\widetilde{K}_t}e^{\frac{y^2}{2}D^2\mu(x_m)}dy+O\left(\frac{1}{\sqrt{t}}\right)}=a(t,x_m)+O\left(\frac{1}{\sqrt{t}}\right),
$$
for $t$ large enough.
\end{proof}
\begin{flushright}
$\blacksquare$
\end{flushright}

\subsection{Convergence to a Dirac mass}
\label{ConvDirac}
In this subsection we prove Proposition \ref{PrinTheo}  $(ii)$.\\
\begin{proof}(ii)\\
We begin by defining
$$
f(t,x)=\frac{n(t,x)}{\rho(t)}=\frac{p_0(t,x)e^{\mu(x)t}}{\dis\int_{\R^d} p_0(t,x)e^{\mu(x)t}dx}.
$$
Therefore, since $\dis\int_{\R^d}f(t,x)dx=1$, there exists a sub-sequence $(f_{t_k})$ that converges weakly to a measure $\nu$, i.e.
$$
\int_{\R^d} f_{t_k}\varphi dx\rightarrow\int_{\R^d} \nu\varphi dx\qquad
\forall\; \varphi\in C_c(\R^d).
$$
We first prove
\begin{equation}
\label{SubseqConv}
\int_{\Omega_{\zeta}^c}\frac{n(t_k,x)}{\rho(t_k)}\varphi(x)dx\longrightarrow 0 \quad \mathrm{as}\quad t_k\rightarrow\infty \quad \forall\;\varphi: \mathrm{supp}\;\varphi\subset \Omega_{\zeta}^c,\quad\mathrm{where}\quad\Omega_{\zeta}=\{x\in\R^d:|x-x_m|<\zeta\}.
\end{equation}
We can rewrite the above integral as below
$$
\int_{\Omega_{\zeta}^c} f(t,x)\varphi(x) dx=\dfrac{1}{\mathcal{I}(t)}\int_{\Omega_{\zeta}^c}p_0(t,x)e^{\mu(x)t}\varphi(x)dx ,
$$
where
$$
\mathcal{I}(t)=\int_{\R^d}p_0(t,y)e^{\mu(y)t}dy.
$$
We estimate $\mathcal{I}(t)$ using the Laplace's method for integration and the assumption \eqref{Hessian}. It follows
$$
\mathcal{I}(t)\sim \frac{e^{t\mu(x_m)}p_0(t,x_m)}{\sqrt{|\det H|}}\left(\frac{2\pi}{t}\right)^{d/2}\qquad \mathrm{as}\quad t\rightarrow\infty,
$$
with $\mu(x_m)$ the strict maximum that is attained at a single point thanks to assumption \eqref{x_m}, and $H$ given by \eqref{Hessian}. Since $p_0(t,x)$ is positive and periodic with respect to $t$, there exist positive constants $K_1$, $K_2$ such that
$$
K_1\frac{e^{t\mu(x_m)}}{t^{d/2}}\leq \mathcal{I}(t)\leq K_2\frac{e^{t\mu(x_m)}}{t^{d/2}}.
$$
Next we note that
$$
t^{d/2}e^{-t\mu(x_m)}\int_{\Omega_{\zeta}^c}p_0(t,x)e^{\mu(x)t}\varphi(x)dx\longrightarrow0,\quad\text{as}\ t\rightarrow+\infty,
$$
since $\mu(x)-\mu(x_m)\leq -\beta$ for some $\beta>0$, and $\varphi$ has compact support, which immediately implies \eqref{SubseqConv}.\\
We deduce from \eqref{SubseqConv} by letting $\zeta\rightarrow 0$, that as $t\rightarrow+\infty$ along subsequences
$$
\frac{n(t,x)}{\rho(t)}\rightharpoonup\omega\delta(x-x_m).
$$
We then prove that $\omega=1$, and hence all the sequence converges to the same limit.\\\\
Let $K_R=\{x\in\R^d:|x|\leq R\}$, for $R>0$.\\
We can write using \eqref{model_sigma0} that
\begin{equation}
\label{Exp_n_n0}
n(t,x)=n_0(x)e^{\int_0^t(a(s,x)-\rho(s))ds}.
\end{equation}
Thanks to assumption \eqref{n0_exp} and \eqref{a_neg},  for $R_0\leq R$, by making an analogous analysis to \eqref{Bound_n_outCompact} we obtain
$$
\int_{K_R^c}n(t,x)dx\rightarrow0\quad\text{as}\quad t\rightarrow+\infty.
$$
Moreover, thanks to Section 2.1, we know that $\rho$ converges to $\widetilde{\varrho}$, a periodic and positive function. Therefore, in long time, $\rho$ is bounded from below and above by positive constants. We deduce that
$$
\int_{K_R^c}f(t,x)dx=\int_{K_R^c}\frac{n(t,x)}{\rho(t)}dx\rightarrow0\quad\mathrm{as}\quad t\rightarrow\infty.
$$
Thanks to the above convergence and the fact that $\dis\int_{\R^d} f(t,x)dx=1$, we deduce that $\forall\zeta>0$ there exists a compact set $K$ and $t_0>0$ such that, for all $t\geq t_0$
$$
1-\zeta\leq \int_K f(t,x)dx.
$$
Moreover, we know that $f$ converges weakly to a measure $\omega\delta(x-x_m)$, thus choosing a smooth compactly supported function $\varphi$ such that $\varphi(x)=1$ if $x\in K$, $\varphi(x)=0$ if $x\in (K')^c$ for another compact $K'$ such that $K\varsubsetneq K'$ and $0<\varphi(x)<1$ for $x\in K'\setminus K$, we obtain
$$
\int_{\R^d}f(t,x)\varphi(x)dx=\int_{K}f(t,x)dx+\int_{K^c}f(t,x)\varphi(x)dx\quad\longrightarrow\omega,
$$
where the first term in the RHS is bigger than $1-\zeta$ and the second one is positive. It follows that $1-\zeta\leq \omega$, for all $0<\zeta<1$ and hence $\omega=1$. We conclude that
$$
\frac{n(t,x)}{\rho(t)}\rightharpoonup\delta(x-x_m)\quad\mathrm{as}\quad t\rightarrow+\infty,
$$
which implies, using the convergence result for $\rho$,
$$
n(t,x)-\widetilde{\varrho}(t)\delta(x-x_m)\rightharpoonup 0\quad\mathrm{as}\quad t\rightarrow+\infty,
$$
weakly in the sense of measures.
\begin{flushright}
$\blacksquare$
\end{flushright}
\end{proof}
\section{The case with mutations: long time behavior}
\label{SigmaPositive}
In this section we study \eqref{CompleteModel} with $\sigma>0$ and provide the proof of Proposition \ref{PrinTheoSect2}.\\
To this end, we first introduce a linearized problem. Let $n$ solve \eqref{CompleteModel}, we define $m(t,x)=n(t,x)e^{\int_0^t\rho(s)ds}$ which solves
\begin{equation}
\label{modelLap_in_m}
\left\{
\begin{array}{rrll}
\partial_t m(t,x)-\sigma\Delta m(t,x) &=& m(t,x)a(t,x),&(t,x)\in[0,+\infty)\times\R^d,\\
m(t=0,x)&=&n_0(x),
\end{array}
\right.
\end{equation}
and associate to \eqref{modelLap_in_m} the parabolic eigenvalue problem \eqref{EigenPb}. In subsection 3.1, we provide a convergence result for \eqref{modelLap_in_m}. Next, using this property, we prove Proposition  \ref{PrinTheoSect2} in subsection 3.2.
\subsection{A convergence result for the linearized problem}
In this section we provide a convergence result for the linearized problem.
\begin{lemma}
\label{existence_eigen}
Assume \eqref{a_W3inf}, \eqref{n0_exp} and \eqref{a_lambda_neg}. Then,
\begin{itemize}
\item[(i)] there exists a unique principal eigenpair $(\lambda,p)$ for the problem \eqref{EigenPb}, with $p\in L^\infty(\R\times\R^d)$, up to nor\-ma\-li\-za\-tion of $p$. Moreover, the eigenfunction $p(t,x)$ is exponentially stable, i.e. there exist a constant $\alpha>0$ such that the solution $\overline{m}(t,x)$ to problem \eqref{modelLap_in_m} satisfies
\begin{equation}
\label{App_m}
\Vert \overline{m}(t,x)e^{\lambda t}-\alpha p (t,x)\Vert_{L^{\infty}(\R^d)}\rightarrow 0\quad\mathrm{as}\;t\rightarrow\infty,
\end{equation} 
exponentially fast.
\item[(ii)] Moreover, let $\delta$ and $R_0$ given by \eqref{a_lambda_neg}, then we have
\begin{equation}
\label{BoundSup_p}
p(t,x)\leq \Vert p\Vert_{L^\infty}e^{-\sqrt{\frac{\delta}{\sigma}}(|x|-R_0)},\;\quad \forall(t,x)\in[0,+\infty)\times\R^d.
\end{equation}
\end{itemize}
\end{lemma}
\begin{proof}
\textit{The proof of (i).}\\ We will apply a result from \cite{huska08} to equation
\begin{equation}
\label{m_tilda}
\partial_t\widetilde{m}-\sigma\Delta \widetilde{m}=\widetilde{m}[a(t,x)+\lambda+\delta],\quad (t,x)\in\R\times\R^d,
\end{equation}
with $\delta$ given in assumption \eqref{a_lambda_neg}. This result allows to show that there exists a unique principal eigenpair for the equation \eqref{m_tilda}, with an eigenfunction which is exponentially stable.\\
Consider the problem
\begin{equation}
\label{LambdaR}
\left\{\begin{array}{rllc}
\partial_t\widetilde{\phi}_R-\sigma\Delta\widetilde{\phi}_R&=&\widetilde{\phi}_R[a+\lambda+\delta],&\mathrm{in}\;\R\times B_R,\\
\widetilde{\phi}_R&=&0,& \mathrm{on}\;\R\times\partial B_R.\\
\end{array}
\right.
\end{equation}
Thanks to \eqref{a_W3inf} and \eqref{a_lambda_neg} we can choose $R$ and $\delta>0$ such that there exists $d_\delta>0$
$$
\Vert a(t,x)+\lambda+\delta\Vert_{L^{\infty}([0,+\infty)\times B_R)}<d_\delta,\quad a(t,x)+\lambda+\delta<0, \quad\forall\;|x|\geq R_0.
$$
Note that $\widetilde{\phi}_R=p_Re^{t(\delta -\lambda_R+\lambda)}$ is a positive  entire solution  to \eqref{LambdaR}. Moreover, it satisfies the hy\-po\-the\-sis (H1) of \cite{huska08}, that is
$$
\dfrac{\Vert \widetilde{\phi}_R(t,\cdot)\Vert_{L^{\infty}(B_R)}}{\Vert \widetilde{\phi}_R(s,\cdot)\Vert_{L^{\infty}(B_R)}}=\dfrac{\Vert p_R(t,\cdot)\Vert_{L^{\infty}(B_R)}}{\Vert p_R(s,\cdot)\Vert_{L^{\infty}(B_R)}}e^{(\delta-\lambda_R+\lambda)(t-s)}\geq Ce^{(\delta-\lambda_R+\lambda)(t-s)},\quad t\geq s,
$$
with $\delta-\lambda_R+\lambda>0$ for $R$ large enough.\\
Therefore Theorem 2.1 (and its generalization Theorem 9.1) in \cite{huska08} implies that there exists a unique positive entire solution $\widetilde{\phi}$ for problem \eqref{m_tilda}, which is given by
$$
\widetilde{\phi}(t,x)=\lim_{R\rightarrow\infty}\widetilde{\phi}_R(t,x).
$$
Moreover, for $p=\widetilde{\phi}e^{-\delta t}$ we obtain
$$
p(t,x)=\lim_{R\rightarrow\infty}p_R(t,x),
$$
and  since $p_R$ is the solution of \eqref{EqLamdaR}, then $p$ is a positive periodic eigenfunction to \eqref{EigenPb}.\\
Furthermore, Theorem 2.2 in \cite{huska08} implies also that
$$
\frac{\Vert\widetilde{m}(t,x)-\alpha \widetilde{\phi}(t,x)\Vert_{L^\infty(\R^d)}}{\Vert \widetilde{\phi}(t,\cdot)\Vert_{L^{\infty}(\R^d)}}\longrightarrow0,
$$
exponentially fast as $t\rightarrow\infty$.\\
Noting that every solution $m$ of problem \eqref{modelLap_in_m} can be written as $m=\widetilde{m}e^{-\lambda t-\delta t}$, we obtain
$$
\Vert m(t,x)e^{\lambda t}-\alpha p(t,x)\Vert_{L^\infty(\R^d)}\longrightarrow0\quad\;\mathrm{as}\quad t\rightarrow+\infty,
$$
and this convergence is also exponentially fast.\\\\
\textit{The proof of (ii).}\\
Next we prove \eqref{BoundSup_p} following similar arguments as in the proof of Lemma 2.4 in \cite{Polacik}. Let $\widetilde{a}(t,x)=a(t,x)+\lambda$ then $p$ is a positive bounded solution of the following equation
\begin{equation}
\label{a_tilda}
\partial_t p-\sigma\Delta p=p\widetilde{a}(t,x),\quad\mathrm{in}\;\R\times \R^d.
\end{equation}
Note that we have defined $p$ in $(-\infty,0]$ by periodic prolongation. Let $\Vert p\Vert_{L^\infty(\R\times\R^d)}=M$. We define
$$
\zeta(t,x)=Me^{-\delta(t-t_0)}+Me^{-\nu(|x|-R_0)},
$$
where $\nu=\sqrt{\frac{\delta}{\sigma}}$ and $R_0$ is given by \eqref{a_lambda_neg}. One can verify that 
$$
M\leq \zeta(t,x)\quad\mathrm{if}\;|x|=R_0\;\mathrm{or}\;t=t_0.
$$
Furthermore if  $|x|>R_0$ or $t>t_0$ evaluating in \eqref{a_tilda} shows
$$
\partial_t\zeta-\sigma\Delta\zeta-\zeta\widetilde{a}(t,x)=Me^{-\delta(t-t_0)}(-\delta-\widetilde{a}(t,x))+Me^{-\nu(|x|-R_0)}\left(-\sigma\nu^2-\widetilde{a}(t,x)+\sigma\nu\frac{d-1}{|x|}\right)\geq 0,
$$
since $\widetilde{a}(t,x)\leq -\delta$ thanks to assumption \eqref{a_lambda_neg}. Thus $\zeta$ is a supersolution of \eqref{a_tilda} on
$$
Q_{0}=\{(t,x)\in (t_0,\infty)\times \R^d\;;|x|>R_0\},
$$
which dominates $p$ on the parabolic boundary of $Q_{0}$. Applying the maximum principle to $\zeta-p$, we obtain
$$
p(t,x)\leq Me^{-\delta(t-t_0)}+Me^{-\nu(|x|-R_0)},\qquad |x|\geq R_0,\;t\in(t_0,\infty).
$$
Taking the limit $t_0\rightarrow-\infty$ yields
$$
p(t,x)\leq Me^{-\nu(|x|-R_0)},\qquad |x|\geq R_0,\;t\leq+\infty,
$$
for $\nu=\sqrt{\frac{\delta}{\sigma}}$. We conclude that $p$ satisfies \eqref{BoundSup_p}.
\begin{flushright}
$\blacksquare$
\end{flushright}
\end{proof}

\subsection{The proof of Proposition \ref{PrinTheoSect2}}
To prove Proposition \ref{PrinTheoSect2} we first prove the following Lemmas.
\begin{lemma}
\label{bound_n}
Assume \eqref{a_W3inf} and \eqref{n0_exp} and let $C_3=\sigma C_2^2+d_0$ then the solution $n(t,x)$ to equation \eqref{CompleteModel} satisfies
$$
n(t,x)\leq \exp\left(C_1-C_2|x|+C_3t\right),\quad\forall(t,x)\in(0,+\infty)\times\R^d.
$$	 
\end{lemma}
\begin{proof}
Define the function $\widetilde{n}(t,x)=\exp\left(C_1-C_2|x|+C_3t\right)$.\\
We prove that $n\leq\widetilde{n}$. To this end we proceed by a comparison argument. One can easily verify that for $C_3$ defined above, we have the following inequality
$$
\partial_t\widetilde{n} -\sigma\Delta \widetilde{n}-\left[a(t,x)+\rho(t)\right]\widetilde{n}
=e^{\left(C_1-C_2|x|+C_3t\right)}\left[C_3-\sigma C_2^2+\sigma\frac{C_2(d-1)}{|x|}-a(t,x)+\rho(t)\right]
\geq 0,\quad \text{a.e in } \R\times\R^d.
$$
Moreover, we have for $t=0$, $n(0,x)\leq \widetilde{n}(0,x)$ thanks to assumption \eqref{n0_exp}. We can then apply a Maximum Principle, in the class of $L^2$ functions, and we conclude that
$$
n(t,x)\leq\widetilde{n}(t,x),\quad\forall(t,x)\in(0,+\infty)\times\R^d.
$$
\begin{flushright}
$\blacksquare$
\end{flushright}
\end{proof}
\begin{lemma}
\label{TechProp}
Assume \eqref{a_W3inf}, \eqref{n0_exp} and \eqref{a_lambda_neg} then 
$$
\left\vert\int_{\R^d}\frac{n(t,x)a(t,x)}{\rho(t)}dx -Q(t)\right\vert\longrightarrow 0,\quad\text{as }t\rightarrow+\infty,
$$
\end{lemma}
with $Q(t)$ given by \eqref{PyQ}.\\
\begin{proof}
From \eqref{App_m}, we obtain that 
$$
n(t,x)e^{\int_0^t\rho(s)ds+\lambda t}=\alpha p(t,x)+\Sigma(t,x),
$$
with $\Vert\Sigma(t,x)\Vert_{L^\infty}\rightarrow0$ exponentially fast, as $t\rightarrow\infty$.\\
We define the compact set $K_t=\{x\in\R^d:|x|\leq At\}$, for some $A>>1$ large enough and compute
$$
\frac{1}{\rho(t)}\int_{\R^d}n(t,x)a(t,x)dx=\dfrac{\dis\int_{K_t}\alpha p(t,x)a(t,x)dx+\int_{K_t}\Sigma(t,x) a(t,x)dx+\int_{K_t^c}\left(\alpha p(t,x)+\Sigma(t,x)\right)a(t,x)\;dx}{\dis\int_{K_t}\alpha p(t,x)dx+\int_{K_t}\Sigma(t,x) dx+\int_{K_t^c}\left(\alpha p(t,x)+\Sigma(t,x)\right)\;dx}.
$$
We then notice that 
$$
\left\vert\dis\int_{K_t}\Sigma(t,x) a(t,x)dx\right\vert\leq \Vert a\Vert_{L^\infty}\Vert\Sigma(t,\cdot)\Vert_{L^\infty}|K_t|\rightarrow 0\quad\mathrm{as}\;t\rightarrow\infty,
$$
since $\Vert\Sigma(t,\cdot)\Vert_{L^\infty}$ converges exponentially fast to zero and the measure of $K_t$ is at most algebraic in $t$. Making the same analysis for $\left\vert\int_{K_t}\Sigma(t,x)dx\right\vert$ it will just remain to prove that the integral terms taken outside the compact set $K_t$ vanish as $t\rightarrow+\infty$.\\
We have, trivially
$$
\left\vert\dis\int_{K_t^c}(\alpha p(t,x)+\Sigma(t,x))a(t,x)dx\right\vert\leq \Vert a\Vert_{L^\infty}\left\vert\int_{K_t^c}n(t,x)e^{\int_0^t(\rho(s)+\lambda)ds}dx\right\vert.
$$
Then we use Lemma \ref{bound_n} to obtain
$$
\int_{K_t^c}n(t,x)e^{\int_0^t(\rho(s)+\lambda)ds}dx\leq \int_{K_t^c}e^{C_1-C_2|x|+Mt}dx\leq e^{C_1+M t}\int_{K_t^c}e^{-C_2|x|}dx\rightarrow 0,\quad\mathrm{as}\;t\rightarrow+\infty.
$$
for $A>M$ large enough, where $M\geq \rho_M+\lambda+C_3$.\\
Combining the last two inequalities we obtain that the integral terms taken outside the compact, vanish as $t\rightarrow+\infty$. This concludes the proof.
\begin{flushright}
$\blacksquare$
\end{flushright}
\end{proof}
\textbf{Proof of Proposition \ref{PrinTheoSect2}}\\\\
\textbf{Convergence of $\rho$.}\\
By integrating equation \eqref{CompleteModel} in $x$, we obtain that
$$
\int_{\R^d}\partial_t n(t,x)dx=\int_{\R^d}n(t,x)[a(t,x)-\rho(t)]dx,
$$
and using Lemma \ref{TechProp} we deduce that
$$
\dfrac{d\rho}{dt}=\rho(t)\left[\dis\int_{\R^d}\frac{n(t,x)a(t,x)}{\rho(t)}dx-\rho(t)\right]=\rho(t)\left[Q(t)+\Sigma'(t)-\rho(t)\right],
$$
where $\Sigma'(t)\rightarrow0$ exponentially as $t\rightarrow\infty$, and $Q(t)$ is given by \eqref{PyQ}.\\
Following similar arguments as in the proof of Proposition \ref{PrinTheo} we obtain $\vert\rho(t)-\widetilde{\rho}(t)\vert\rightarrow0$ as $t\rightarrow\infty$, with $\widetilde{\rho}$ the unique solution of 
$$
\left\{\begin{array}{l}
\dfrac{d\widetilde{\rho}}{dt}=\widetilde{\rho}(t)\left[Q(t)-\widetilde{\rho}(t)\right],\\
\widetilde{\rho}(0)=\widetilde{\rho}(T),
\end{array}\right.
$$
provided $\dis\int_0^TQ(t)dt>0$. Moreover if $\dis\int_0^TQ(t)dt\leq0$, then $\rho(t)\rightarrow0$ as $t\rightarrow\infty$. Note also that 
\begin{equation}
\label{lambda}
\lambda=-\frac{1}{T}\int_0^TQ(t)dt.
\end{equation}
We consider, indeed, the eigenvalue problem \eqref{EigenPb} and integrate it in $x\in\R^d$
$$
\partial_t\int_{\R^d}p(t,x)dx=\int_{\R^d}a(t,x)p(t,x)dx+\lambda\int_{\R^d}p(t,x)dx.
$$
We divide by $\dis\int_{\R^d}p(t,x)dx$ and integrate now in $t\in[0,T]$, to obtain
$$
\dis\int_0^T\dfrac{\partial_t\int_{\R^d}p(t,x)dx}{\int_{\R^d}p(t,x)dx}dt=\int_0^TQ(t)dt+\lambda T,
$$
which implies that
$$
0=\ln\left(\int_{\R^d}p(T,x)dx\right)-\ln\left(\int_{\R^d}p(0,x)dx\right)=\int_0^TQ(t)dt+\lambda T,
$$
and hence \eqref{lambda}.\\
This ends the proof of statements $(i)-(ii)$ of Proposition \ref{PrinTheoSect2}.\\\\
\textbf{Convergence of $\dfrac{n}{\rho}$.}\\
Let $K_t=\{x\in\R^d:|x|<At\}$, for $A>R_0$, as in the proof of Lemma \ref{TechProp}, we can write
$$
\begin{array}{rll}
\dfrac{n(t,x)}{\rho(t)}&=&\dfrac{\alpha p(t,x)+\Sigma(t,x)}{\dis\int_{K_t}\left(\alpha p(t,x)+\Sigma(t,x)\right)dx+\int_{K_t^c}\left(\alpha p(t,x)+\Sigma(t,x)\right)dx}.
\end{array}
$$
Following similar arguments as in Lemma \ref{TechProp} we obtain that 
$$
\left\Vert \dfrac{n(t,x)}{\rho(t)}-P(t,x)\right\Vert_{L^{\infty}}\longrightarrow0,
$$ 
with $P(t,x)$ as in \eqref{PyQ}.\\
Consequently, when $\lambda<0$ we obtain that 
$$
\Vert n(t,\cdot)-\widetilde{\rho}(t)P(t,\cdot)\Vert_{L^\infty}\longrightarrow 0\qquad\mathrm{as}\;t\rightarrow\infty,
$$
and this concludes the proof of $(iii)$.
\begin{flushright}
$\blacksquare$
\end{flushright}
\section{Case $\sigma<<1$. Small mutations}
\label{SmallMutations}
In this section we choose $\sigma=\varepsilon^2$ and we prove that for $\varepsilon$ small enough, the principal eigenvalue $\lambda$ given in \eqref{EigenPb} is negative. As a consequence, thanks to Proposition \ref{PrinTheoSect2}, any solution of \eqref{Pb_n_Epsilon1} converges to the unique periodic solution $(n_\varepsilon,\rho_\varepsilon)$. Next, we prove Theorem \ref{Prin_Theo_ueps}, which allows to characterize $n_\varepsilon$, as $\varepsilon\rightarrow0$.\\
Consider now the problem \eqref{Pb_n_Epsilon1}
and let $(\lambda_\varepsilon,p_\varepsilon)$ be the eigenelements of problem \eqref{EigenPb} for $\sigma=\varepsilon^2$, then we have the following result.
\begin{lemma}
\label{lam_eps_neg}
Under assumption \eqref{x_m} there exists $\lambda_m>0$ and $\varepsilon_0>0$ such that for all $\varepsilon<\varepsilon_0$ we have $\lambda_\varepsilon\leq-\lambda_m$.
\end{lemma}
\begin{proof}
We follow the proof for the case of bounded domains, given in \cite{hess}.\\
For $R>0$ define $B_R:=B(x_m,R)$ and $a_R(t)=\dis\min_{x\in\overline{B}_R}a(t,x)$. Then we choose $R_1$ small enough such that 
\begin{equation}
\label{int_a_R0}
\int_0^Ta_{R_1}(t)dt>0.
\end{equation}
This is possible thanks to \eqref{x_m} and the continuity of $a$ with respect to $x$.\\
We first consider the periodic-parabolic Dirichlet eigenvalue problem on $[0,+\infty)\times B_{R_1},$
\begin{equation}
\label{Model_w_eps}
\left\{
\begin{array}{cr}
\partial_t \underline{w}-\varepsilon^2 \Delta\underline{w}-a_{R_1}(t)\underline{w}=\underline{\mu_\varepsilon}\underline{w},&\mathrm{in}\;[0,+\infty)\times B_{R_1},\\
\underline{w}=0, &\mathrm{on}\; [0,+\infty)\times\partial B_{R_1},\\
\underline{w}:\;T-\mathrm{periodic\;in}\;t.
\end{array}
\right.
\end{equation}
We calculate $\underline{\mu_\varepsilon}$ by the Ansatz $\underline{w}(t,x)=\alpha(t)\varphi_1(x)$ where $\varphi_1>0$ is the principal eigenfunction of 
$$
\left\{\begin{array}{rclr}
-\Delta\varphi_1&=&\gamma_1\varphi_1,&\mathrm{in}\;B_{R_1},\\
\varphi_1&=&0,&\mathrm{on}\;\partial B_{R_1},\\
\end{array}
\right.
$$
with principal eigenvalue $\gamma_1>0$. By substituting in \eqref{Model_w_eps} we deduce that
$$
\alpha'(t)\varphi_1+\gamma_1\varepsilon^2\alpha(t)\varphi_1-a_{R_1}(t)\alpha(t)\varphi_1=\underline{\mu_\varepsilon}\alpha(t)\varphi_1,
$$
and consequently
$$
\alpha(t)=\alpha(0)\exp\left(\int_0^ta_{R_1}(\tau)d\tau-(\gamma_1\varepsilon^2-\underline{\mu_\varepsilon})t\right).
$$
For $\underline{w}$ to be $T-$periodic we must have
$$
\int_0^Ta_{R_1}(\tau)d\tau-\gamma_1\varepsilon^2 T+T\underline{\mu_\varepsilon}=0.
$$
We deduce indeed that choosing
$$
\underline{\mu_\varepsilon}=\gamma_1\varepsilon^2-\frac{1}{T}\int_0^Ta_{R_1}(t)dt,
$$
we obtain the principal eigen-pair $(\underline{w},\underline{\mu_\varepsilon})$ for \eqref{Model_w_eps}. Next we consider the periodic-parabolic eigenvalue problem
$$
\left\{
\begin{array}{cr}
\partial_t w-\varepsilon^2 \Delta w-a(t,x)w=\lambda_{R_1}w,&\mathrm{in}\;[0,+\infty)\times B_{R_1},\\
w=0, &\mathrm{on}\;[0,+\infty)\times \partial B_{R_1},\\
w:\;T-\mathrm{periodic\;in}\;t.
\end{array}
\right.
$$
Since $a_{R_1}(t)\leq a(t,x)$ on $[0,+\infty)\times B_{R_1}$ we have $\lambda_{R_1}\leq\underline{\mu_\varepsilon}$ (Lemma 15.5 \cite{hess}). By monotony of eigenvalues with respect to the domain we obtain $\lambda_\varepsilon\leq\lambda_{R_1}\leq\underline{\mu_\varepsilon}$. Finally, thanks to \eqref{int_a_R0} we conclude that there exist $\lambda_m>0$, $\varepsilon_0>0$ such that for all $\varepsilon\leq\varepsilon_0$, we have $\lambda_\varepsilon\leq-\lambda_m$.
\begin{flushright}
$\blacksquare$
\end{flushright}
\end{proof}
In the following subsections we provide the proof of Theorem \ref{Prin_Theo_ueps}. In Subsection \ref{Unif_Bounds}, we give some global bounds for $\rho_\varepsilon$. Next, in Subsection \ref{regularity_u}, we prove that $(u_\varepsilon)$ is locally uniformly bounded, Lipschitz with respect to $x$ and locally equicontinuous in time. In the last subsection we conclude the proof of Theorem \ref{Prin_Theo_ueps} letting $\varepsilon$ goes to zero and describing the limits of $u_\varepsilon$, $n_\varepsilon$ and $\rho_\varepsilon$. 

\subsection{Uniform bounds for $\rho_\varepsilon$}
\label{Unif_Bounds}
In this section we provide uniform bounds for $\rho_\varepsilon$.

\begin{lemma}
\label{Borns_rho_Lemma}
Assume \eqref{a_W3inf}, \eqref{a_lambda_neg}. Then for every $\varepsilon>0$, there exist positive constants $\rho_m$ and $\rho_M$ such that
\begin{equation}
\label{Borns_rho_eps}
0<\rho_m\leq\rho_\varepsilon(t)\leq \rho_M\quad\forall t\geq 0.
\end{equation}
\end{lemma}
\begin{proof}
From equation \eqref{Pb_n_Epsilon1} integrating in $x\in\R^d$ and using assumption \eqref{a_W3inf} we get
\begin{equation}
\label{bornDerRho}
\frac{d\rho_\varepsilon}{dt}=\int_{\R^d}n_\varepsilon(t,x)[a(t,x)-\rho_\varepsilon(t)dx]\leq \rho_\varepsilon(t)[d_0-\rho_\varepsilon(t)].
\end{equation}
This implies that 
$$
\rho_\varepsilon(t)\leq\rho_M:=\max(\rho_\varepsilon^0,d_0).
$$
To obtain the lower bound we recall that $n_\varepsilon(t,x)=\rho_\varepsilon(t)P_\varepsilon(t,x)$, with $\rho_\varepsilon(t)$ the unique periodic solution of
$$
\frac{d\rho_\varepsilon}{dt}=\rho_\varepsilon(t)[Q_\varepsilon(t)-\rho_\varepsilon(t)],
$$
and with $Q_\varepsilon(t)$ and $P_\varepsilon(t,x)$ given by \eqref{PyQ}.
From \eqref{Exp_rho} we know that
\begin{equation}
\label{Exp_rho_eps}
\rho_\varepsilon(t)=\frac{1-\exp\left[-\dis\int_0^{T}Q_\varepsilon(s)ds\right]}{\exp\left[-\dis\int_0^{T}Q_\varepsilon(s)ds\right]\dis\int_t^{t+T}\exp\left[\dis\int_t^s Q_\varepsilon(\theta)d\theta\right]ds}.
\end{equation}
From Lemma \ref{lam_eps_neg}, we note that, $\lambda_\varepsilon=-\dfrac{1}{T}\dis\int_0^TQ_\varepsilon(t)dt\leq -\lambda_m$, thus
$$
\exp\left[-\dis\int_0^{T}Q_\varepsilon(\theta)d\theta\right]\leq e^{-T\lambda_m}.
$$
Also from \eqref{a_W3inf} and \eqref{PyQ} we get
$$
\int_t^{t+T}\exp\left[\dis\int_t^s Q_\varepsilon(\theta)d\theta\right]ds\leq\int_t^{t+T}e^{d_0T}ds=Te^{d_0T}. 
$$
Combining the above inequalities with \eqref{Exp_rho_eps} we obtain
$$
0< \rho_m:=\frac{1}{T}e^{-d_0T}\left(e^{\lambda_mT}-1\right)\leq \rho_\varepsilon(t),\quad \forall\;t\geq0.
$$
\begin{flushright}
$\blacksquare$
\end{flushright}
\end{proof}

\subsection{Regularity results for $u_\varepsilon$}
\label{regularity_u}
In this section we study the regularity properties of $u_\varepsilon =\varepsilon\ln \big((2\pi\varepsilon)^{d/2}n_\varepsilon\big)$, where $n_\varepsilon$ is the unique periodic solution of equation \eqref{Pb_n_Epsilon1}.\\
\begin{theorem}
Assume \eqref{a_W3inf}, \eqref{x_m} and \eqref{a_lambda_neg}. Then $u_\varepsilon$ is locally uniformly bounded and locally equicontinuous in time in $[0,T]\times \R^d$. Moreover, for some $D>0$, $\omega_\varepsilon=\sqrt{2D-u_\varepsilon}$, is Lipschitz continuous with respect to $x$ in $(0,\infty)\times\R^d$ and there exists a positive constant $C$ such that we have the following
\begin{equation}
\label{Nabla_omega}
\left|\nabla \omega_\varepsilon\right|\leq C,\quad \text{in }[0,+\infty)\times\R^d,
\end{equation}
\begin{equation}
\label{Equi_u}
\forall R>0,\ \sup_{t\in[0,T],\ x\in B_R}|u_\varepsilon(t,x)-u_\varepsilon(s,x)|\rightarrow0\quad\mathrm{as }\:\varepsilon\rightarrow 0.
\end{equation}
\end{theorem}
We prove this theorem in several steps. 
\subsubsection{An upper bound for $u_\varepsilon$}
We recall from \eqref{convPerio} that $n_\varepsilon(t,x)=\rho_\varepsilon(t)\dfrac{p_\varepsilon(t,x)}{\int_{\R^d}p_\varepsilon(t,x)dx}$, where
\begin{equation}
\label{p_eps}
\left\{
\begin{array}{lr}
\partial_t p_\varepsilon-\varepsilon^2\Delta p_\varepsilon-a(t,x)p_\varepsilon=\lambda_\varepsilon p_\varepsilon,&\mathrm{in}\;\R\times \R^d,\\
0<p_\varepsilon:\ T-\mathrm{periodic},\\
\Vert p_\varepsilon(0,x)\Vert_{L^\infty(\R^d)}=1.
\end{array}
\right.
\end{equation}
Define $q_\varepsilon(t,x)=p_\varepsilon(t,x\varepsilon)$, which satisfies 
\begin{equation}
\label{q_eps}
\left\{
\begin{array}{cr}
\partial_t q_\varepsilon-\Delta q_\varepsilon=a_\varepsilon(t,x)q_\varepsilon,&\mathrm{in}\;\R\times \R^d,\\
q_\varepsilon:\quad T-\mathrm{periodic},
\end{array}
\right.
\end{equation}
for $a_\varepsilon(t,x)=a(t,x\varepsilon)+\lambda_\varepsilon$. Note that $a_\varepsilon$ is uniformly bounded thanks to the $L^\infty-$norm of $a$, which together with Lemma \ref{lam_eps_neg} implies that $-d_0\leq\lambda_\varepsilon\leq-\lambda_m$.\\
Since $\Vert p_\varepsilon(0,x)\Vert_{L^\infty(\R^d)}=1$ we can choose $x_0$ such that $ p_\varepsilon(0,x_0)=1$. Moreover $q_\varepsilon$ is a nonnegative solution of \eqref{q_eps} in $(0,2T)\times B(\frac{x_0}{\varepsilon},1)$. Let $\delta_0$, be such that $0<\delta_0<T$, then we apply the Theorem 2.5 \cite{huska06} which is an elliptic-type Harnack inequality for positive solutions of \eqref{q_eps} in a bounded domain, and we have $\forall\;t\in[\delta_0,2T]$
$$
\sup_{x\in B(\frac{x_0}{\varepsilon},1)}q_\varepsilon(t,x)\leq C\inf_{x\in B(\frac{x_0}{\varepsilon},1)}q_\varepsilon(t,x),
$$
where $C=C(\delta_0,d_0)$. Returning to $p_\varepsilon$ this implies
\begin{equation}
\label{BoundInf_p}
p_\varepsilon(t_0,x_0)\leq \sup_{y\in B(x_0,\varepsilon)}p_\varepsilon(t_0,y)\leq C p_\varepsilon(t_0,x),\quad \forall(t_0,x)\in[\delta_0,2T]\times B(x_0,\varepsilon).
\end{equation}
Since $p_\varepsilon$ is $T-$periodic we conclude that the last inequality is satisfied $\forall\;t\in[0,T]$. 
From  \eqref{Borns_rho_eps}, \eqref{BoundInf_p} and the upper bound \eqref{BoundSup_p} for $p_\varepsilon$ with $\sigma=\varepsilon^2$, we obtain 
$$
n_\varepsilon(0,x)\leq \rho_M \dfrac{p_\varepsilon(0,x)}{\int_{\R^d}p_\varepsilon(0,x)dx}\leq \frac{C\rho_Mp_\varepsilon(0,x)}{\int_{B(x_0,\varepsilon)}p_\varepsilon(0,x_0)dx}\leq \rho_M \dfrac{Cp_\varepsilon(0,x)}{|B(x_0,\varepsilon)|}\leq C'\varepsilon^{-d}\exp^{\frac{C_1'-C_2'|x|}{\varepsilon}},
$$
for all $\varepsilon\leq\varepsilon_0$, with $\varepsilon_0$ small enough, where the constant $C'$ depends on $\rho_M$, $\Vert p\Vert_{L^\infty(R^d)}$ and the constant $C$ in \eqref{BoundInf_p} and $C_1'$ and $C_2'$ depend on the constants of hypothesis \eqref{a_lambda_neg}. Next we proceed with a Maximum Principle argument as in Lemma \ref{bound_n} to obtain for every $(t,x)\in[0,+\infty)\times\R^d$ and $C_3=(C_2')^2+d_0$, 
$$
n_\varepsilon(t,x)\leq C'\exp^{\frac{C_1'-C_2'|x|}{\varepsilon}+C_3 t}.
$$
From here and the periodicity of $u_\varepsilon$, with an abuse of notation for the constants, we can write, for all $\varepsilon\leq\varepsilon_0$
\begin{equation}
\label{Up_Bound_u}
u_\varepsilon(t,x)\leq C_1'-C_2'|x|,\quad \forall(t,x)\in[0,+\infty)\times\R^d.
\end{equation}
\subsubsection{A lower bound for $u_\varepsilon$}
Using the bounds for $a$ in \eqref{a_W3inf} and for $\rho_\varepsilon$ in \eqref{Borns_rho_eps} we obtain for $\widetilde{C}=d_0+\rho_M$
$$
\partial_t n_\varepsilon-\varepsilon^2\Delta n_\varepsilon\geq -\widetilde{C}n_\varepsilon.
$$
Let $n_\varepsilon^*$ be the solution of the following heat equation
$$
\left\{
\begin{array}{l}
\partial_t n_\varepsilon^*-\varepsilon^2\Delta n_\varepsilon^*+\widetilde{C}n_\varepsilon^*=0,\\
n_\varepsilon^*(0,x)=n_\varepsilon^0,
\end{array}
\right.
$$
given explicitly by the Heat Kernel $K$, 
$$
n_\varepsilon^*(t,x)=e^{-\widetilde{C}t}\left(n_\varepsilon^0\ast K\right)=\frac{e^{-\widetilde{C}t}}{\varepsilon^d(4\pi t)^{d/2}}\int_{\R^d}n_\varepsilon^0(y)e^{-\frac{|x-y|^2}{4t \varepsilon^2}}dy,\quad t>0.
$$
By a comparison principle we have $n_\varepsilon^*(t,x)\leq n_\varepsilon(t,x)$. Moreover, from \eqref{BoundInf_p} and \eqref{BoundSup_p} we deduce that
$$ 
\varepsilon^{-d}\widetilde{C}_0e^{-\frac{\widetilde{C}_1}{\varepsilon}}\leq \rho_m \dfrac{p_\varepsilon(0,x)}{\int_{\R^d}p_\varepsilon(0,x)dx}\leq n_\varepsilon(0,x)\quad\forall x\in B(x_0,\varepsilon),
$$
for some positive constants $\widetilde{C}_0$ and $\widetilde{C}_1$ depending on $\Vert p\Vert_{L^\infty}$, $\rho_m$, $\delta$, $d_0$, $d$, and $R_0$, and $x_0$ the point where $p_\varepsilon(0,x_0)=1$. Then
$$
\begin{array}{rcl}
n_\varepsilon(t,x)
&\geq &\dis\frac{\widetilde{C}_0}{\varepsilon^{2d}(4\pi t)^{d/2}}e^{-\frac{\widetilde{C}_1+\varepsilon\widetilde{C}t}{\varepsilon}}\int_{B(x_0,\varepsilon)}e^{-\frac{|x-y|^2}{4t \varepsilon^2}}dy\\
&\geq &
\dis\frac{\widetilde{C}_0|B(x_0,\varepsilon)|}{\varepsilon^{2d}(4\pi t)^{d/2}}e^{-\frac{2|x|^2+2(|x_0|+\varepsilon)^2}{4 t\varepsilon^2}-\frac{\widetilde{C}_1+\widetilde{C}t\varepsilon}{\varepsilon}}.
\end{array}
$$
This, together with the definition of $u_\varepsilon$, implies that
$$
\varepsilon\log\left(\frac{\widetilde{C} _0|B(x_0,\varepsilon)|}{\varepsilon^{2d}(4\pi t)^{d/2}}\right)-\frac{|x|^2+(|x_0|+\varepsilon)^2}{2t\varepsilon}-(\widetilde{C}_1+\widetilde{C}t\varepsilon)\leq u_\varepsilon(t,x),\quad\forall t\geq 0.
$$
In particular, we obtain that
$$
\varepsilon\log\left(\frac{\widetilde{C} _0|B(x_0,\varepsilon)|}{\varepsilon^{3d/2}(4\pi t)^{d/2}}\right)-\frac{|x|^2+(|x_0|+\varepsilon)^2}{2t}-(\widetilde{C}_1+\widetilde{C}t)\leq u_\varepsilon(\frac{t}{\varepsilon},x),\quad\forall t\in\left[1,1+\varepsilon T\right],
$$
and again, using the periodicity of $u_\varepsilon$, we obtain a quadratic lower bound for $u_\varepsilon$ for all $t\geq0$; that is, there exist $A_1$, $A_2\geq 0$ and $\varepsilon_0$ such that for all $\varepsilon\leq\varepsilon_0$,

\begin{equation}
\label{Low_Bound_u}
-A_1|x|^2-A_2\leq u_\varepsilon(t,x) .
\end{equation}

\subsubsection{Lipschitz bounds}
In this section we prove \eqref{Nabla_omega}. To this end we use a Bernstein type method closely related to the one used in \cite{MAA}. Let $\omega_\varepsilon = \sqrt{2C_1'-u_\varepsilon}$,  for $C_1'$ given by \eqref{Up_Bound_u}, thus $\omega_\varepsilon$ satisfies
$$
\frac{1}{\varepsilon}\partial_t\omega_\varepsilon-\varepsilon\Delta\omega_\varepsilon-\left(\frac{\varepsilon}{\omega_\varepsilon}-2\omega_\varepsilon\right)|\nabla\omega_\varepsilon|^2=\frac{a(t,x)-\rho_\varepsilon(t)}{-2\omega_\varepsilon}.
$$
Define $W_\varepsilon=\nabla \omega_\varepsilon$, which is also $T-$periodic. We differentiate the above equation with respect to $x$ and multiply by $\frac{W_\varepsilon}{|W_\varepsilon|}$, i.e
$$
\frac{1}{\varepsilon}\partial_t|W_\varepsilon|-\varepsilon\Delta|W_\varepsilon|-2\left(\frac{\varepsilon}{\omega_\varepsilon}-2\omega_\varepsilon\right)W_\varepsilon\cdot\nabla|W_\varepsilon|+\left(\frac{\varepsilon}{\omega_\varepsilon^2}+2\right)|W_\varepsilon|^3\leq\frac{\left(a(t,x)-\rho_\varepsilon(t)\right)|W_\varepsilon|}{2\omega_\varepsilon^2}-\frac{\nabla a\cdot W_\varepsilon}{2\omega_\varepsilon|W_\varepsilon|}.
$$
From \eqref{Up_Bound_u} we know that $u_\varepsilon\leq C_1'$, which together with \eqref{Low_Bound_u} implies
$$
\sqrt{C_1'}\leq \omega_\varepsilon\leq \sqrt{2C_1'+A_1|x|^2+A_2}.
$$
It follows that
$$
\left|2\left(\frac{\varepsilon}{\omega_\varepsilon}-2\omega_\varepsilon\right)\right|\leq A_4|x|+C_4,
$$
for some constants $A_4$ and $C_4$, from where, we have for $\theta$ large enough
\begin{equation}
\label{wp_eps_theta}
\frac{1}{\varepsilon}\partial_t|W_\varepsilon|-\varepsilon\Delta|W_\varepsilon|-\big(A_4|x|+C_4\big)\big\vert W_\varepsilon\cdot\nabla|W_\varepsilon|\big\vert+2\left(|W_\varepsilon|-\theta\right)^3\leq0.
\end{equation}
Let $T_M>2T$ and $A_5$ to be chosen later, define now, for $(t,x)\in\Big(0,\frac{T_M}{\varepsilon}\Big]\times B_R(0)$
$$
\overline{W}(t,x)=\frac{1}{2\sqrt{t\varepsilon}}+\frac{A_5R^2}{R^2-|x|^2}+\theta.
$$
We next verify that $\overline{W}$ is a strict supersolution of \eqref{wp_eps_theta} in $\Big(0,\frac{T_M}{\varepsilon}\Big]\times B_R(0)$. To this end we compute
$$
\partial_t \overline{W}=-\frac{1}{4t\sqrt{t\varepsilon}},\quad \nabla \overline{W}=\frac{2A_5R^2x}{(R^2-|x|^2)^2},\quad \Delta \overline{W}=\frac{2A_5R^2d}{(R^2-|x|^2)^2}+\frac{8A_5R^2|x|^2}{(R^2-|x|^2)^3},
$$
and then replace in \eqref{wp_eps_theta} to obtain
$$
\begin{array}{l}
\frac{1}{\varepsilon}\partial_t\overline{W}-\varepsilon\Delta \overline{W}-\big(A_4|x|+C_4\big)|\overline{W}\nabla \overline{W}|+2\left(\overline{W}-\theta\right)^3\\\\

=-\frac{1}{4\varepsilon t\sqrt{\varepsilon t}}-\varepsilon\Big[\frac{2A_5R^2d}{(R^2-|x|^2)^2}+\frac{8A_5R^2|x|^2}{(R^2-|x|^2)^3}\Big]-\big(A_4|x|+C_4\big)\left(\frac{1}{2\sqrt{\varepsilon t}}+\frac{A_5R^2}{R^2-|x|^2}+\theta\right)\frac{2A_5R^2|x|}{(R^2-|x|^2)^2}+2\left(\frac{1}{2\sqrt{\varepsilon t}}+\frac{A_5R^2}{R^2-|x|^2}\right)^3\\\\

\geq -\varepsilon \Big[\frac{2A_5R^2d}{(R^2-|x|^2)^2}+\frac{8A_5R^4}{(R^2-|x|^2)^3}\Big]-\big(A_4R+C_4\big)\left(\frac{1}{2\sqrt{\varepsilon t}}+\frac{A_5R^2}{R^2-|x|^2}+\theta\right)\frac{2A_5R^3}{(R^2-|x|^2)^2}\\\\

\quad+\frac{3A_5R^2}{R^2-|x|^2}\Big(\frac{1}{2t\varepsilon}+\frac{A_5R^2}{\sqrt{\varepsilon t}(R^2-|x|^2)}\Big)+\frac{2A_5^3R^6}{(R^2-|x|^2)^3},
\end{array}
$$
where we have used that $|x|\leq R$. One can verify that the RHS of the above inequality is strictly positive for $R > 1$, $\varepsilon\leq 1$, and $A_5>C\sqrt{T_M}$ for certain  constant $C$ large enough. Therefore, $\overline{W}$ is a strict supersolution of \eqref{wp_eps_theta} in $\Big(0,\frac{T_M}{\varepsilon}\Big] \times B_R(0)$ and for $\varepsilon\leq 1$.\\
We next prove that
$$
|W_\varepsilon(t,x)|\leq \overline{W}(t,x)\quad\mathrm{in}\;\Big(0,\frac{T_M}{\varepsilon}\Big]\times B_R(0).
$$
To this end, we notice that $\overline{W}(t, x)$ goes to $+\infty$ as $|x|\rightarrow R$ or as $t\rightarrow0$. Therefore, $|W_\varepsilon|(t, x) - \overline{W}(t, x)$ attains its maximum at an interior point of $\Big(0,\frac{T_M}{\varepsilon}\Big]\times B_R(0)$. We choose $t_m \leq \frac{T_M}{\varepsilon}$ the smallest time such that the maximum of $|W_\varepsilon|(t, x)-\overline{W}(t, x)$ in the set $(0,t_m]\times B_R(0)$ is equal to 0. If such $t_m$ does not exist, we are done.\\\\
Let $x_m$ be such that $|W_\varepsilon|(t, x)- \overline{W}(t, x)\leq |W_\varepsilon|(t_m, x_m)- \overline{W}(t_m, x_m)= 0$ for all $(t, x) \in(0, t_m)\times B_R(0)$. At such point, we have
$$
0\leq\partial_t\big(|W_\varepsilon|- \overline{W}\big)(t_m, x_m),\quad 0\leq -\Delta\big( |W_\varepsilon|- \overline{W}\big)(t_m, x_m),\quad  |W_\varepsilon|(t_m, x_m)\nabla |W_\varepsilon|(t_m, x_m)= \overline{W}(t_m, x_m))\nabla \overline{W}(t_m, x_m).
$$
Combining the above properties with the facts that $|W_\varepsilon|$ and $\overline{W}$ are respectively sub and strict supersolution of \eqref{wp_eps_theta}, we obtain that
$$
(|W_\varepsilon|(t_m, x_m)-\theta)^3-(\overline{W}(t_m, x_m)-\theta)^3<0\Rightarrow |W_\varepsilon|(t_m, x_m)<\overline{W}(t_m, x_m),
$$
which is in contradiction with the choice of $(t_m, x_m)$. We deduce, then that
$$
|W_\varepsilon(t,x)|\leq \frac{1}{2\sqrt{\varepsilon t}}+\frac{A_5R^2}{R^2-|x|^2}+\theta\quad\mathrm{for}\;(t,x)\in\Big(0,\frac{T_M}{\varepsilon}\Big] \times B_R(0),\;\forall\;R>1.
$$
We note that for $\varepsilon<\varepsilon_0$ small enough we have $\frac{T_M}{\varepsilon}>\frac{2T}{\varepsilon}>\frac{T}{\varepsilon}+T>\frac{T}{\varepsilon}$. Letting $R\rightarrow\infty$ we deduce that
$$
|W_\varepsilon(t,x)|\leq \frac{1}{2\sqrt{\varepsilon t}}+A_5+\theta\leq \frac{1}{2\sqrt{T}}+A_5+\theta\quad\mathrm{for}\;(t,x)\in\left[\frac{T}{\varepsilon},\frac{T}{\varepsilon}+T \right] \times \R^d.
$$
Finally we use the periodicity of $W_\varepsilon$ to extend the result for all $t\in[0,+\infty)$ and we obtain \eqref{Nabla_omega}.
\subsubsection{Equicontinuity in time}
From the above uniform bounds and continuity results we can also deduce uniform equicontinuity in time for the family $u_\varepsilon$ on compact subsets of $]0,+\infty]\times\R^d$ and prove \eqref{Equi_u}. We follow a method introduced in \cite{BarlesEqui}. \\
We will prove that for any $\eta>0$, we can find constants $A$, $B$ large enough such that: for any $x \in B(0, R/2)$, $s\in[0,T]$, and for all $\varepsilon<\varepsilon_0$ we have
\begin{equation}
\label{supersolution}
u_\varepsilon(t,y)-u_\varepsilon(s,x)\leq \eta + A|x-y|^2+\varepsilon B(t-s), \forall (t,y)\in[s,T]\times B_R(0),
\end{equation}
and
\begin{equation}
\label{subsolution}
u_\varepsilon(t,y)-u_\varepsilon(s,x)\geq -\eta- A|x-y|^2-\varepsilon B(t-s), \forall (t,y)\in[s,T]\times B_R(0).
\end{equation}
We provide the proof of \eqref{supersolution}. One can prove \eqref{subsolution} following similar arguments. 
\\Fix $(s, x)$ in
$[0, T[\times B_{R/2}(0)$. Define
$$
\widehat{\xi}(t,y)=u_\varepsilon(s,x) +\eta+ A|x-y|^2+\varepsilon B(t-s), \quad (t,y)\in[s,T[\times B_R(0),
$$
where $A$ and $B$ are constants to be determined. We prove that, for $A$ and $B$ large enough, $\widehat{\xi}$ is a super-solution to \eqref{Pb_u_Epsilon} on $[s,T]\times B_R(0)$ and $\widehat{\xi}(t,y)>u_\varepsilon(t,y)$ for $(t,y)\in \{s\}\times B_R(0)\cup [s,T]\times \partial B_R(0)$.\\
According to the previous section, $\{u_\varepsilon\}_\varepsilon$ is locally uniformly bounded, so we can take a constant $A$ such that for all $\varepsilon < \varepsilon_0$,
$$
\frac{8\Vert u_\varepsilon\Vert_{L^\infty([0,T]\times B_R(0))}}{R^2}\leq A. 
$$
With this choice, $\widehat{\xi}(t, y) > u_\varepsilon(t, y)$ on $[s, T ]\times \partial B_R(0)$, for all $\eta>0$, $B>0$ and $x\in B_{R/2}(0)$. \\
Next we prove that, for $A$ large enough, $\widehat{\xi}(s, y) > u_\varepsilon(s, y)$ for all $y\in B_R(0)$. We
argue by contradiction. Assume that there exists $\eta > 0$ such that for all constants $A$ there exists $y_{A,\varepsilon}\in B_R(0)$ such that
\begin{equation}
\label{Absurd}
u_\varepsilon(s,y_{A,\varepsilon)}-u_\varepsilon(s,x)>\eta+A|y_{A,\varepsilon}-x|^2.
\end{equation}
This implies
$$
|y_{A,\varepsilon}-x|\leq\sqrt{\frac{2M}{A}}\longrightarrow0,\quad\mathrm{as}\; {A\rightarrow\infty}.
$$
Here $M$ is an uniform upper bound for $\Vert u_\varepsilon\Vert_{L^\infty([0,T]\times B_R(0))}$. Then for all $h>0$, there exist $A$ large enough and $\varepsilon_0$ small enough, such that $\forall\varepsilon<\varepsilon_0$,
$$
|y_{A,\varepsilon}-x|\leq h.
$$
Therefore, from the uniform continuity in space of $u_\varepsilon$ taking  $h$ small enough, we obtain
$$
|u_\varepsilon(s,y_{A,\varepsilon})-u_\varepsilon(s,x)|<\eta/2\quad \forall\varepsilon\leq \varepsilon_0,
$$
but this is a contradiction with \eqref{Absurd}. Therefore $\widehat{\xi}(s, y) > u_\varepsilon(s, y)$ for all $y\in B_R(0)$. \\
Finally, noting that $R$ is bounded we deduce that for $B$ large enough, $\widehat{\xi}$ is a super-solution to \eqref{Pb_u_Epsilon} in $[s, T ] \times B_R(0)$. \\
Using a comparison principle, since $u_\varepsilon$ is also a solution of \eqref{Pb_u_Epsilon} we  have
$$
u_\varepsilon(t,y)\leq \widehat{\xi}(t,y)\quad \forall(t,y)\in[s,T]\times B_R(0).
$$
Thus \eqref{supersolution} is satisfied for $t\geq s\geq 0$. Then we put $x = y$ and obtain that for all $\eta>0$ there exists $\varepsilon_0>0$ such that for all $\varepsilon<\varepsilon_0$
$$
|u_\varepsilon(t,x)-u_\varepsilon(s,x)|\leq \eta + \varepsilon B(t-s),
$$
for every $(t,x)\in[0,T]\times B_R(0)$.
This implies that $u_\varepsilon$ is locally equicontinuous in time. Moreover letting $\varepsilon\rightarrow0$ we obtain \eqref{Equi_u}.

\subsection{Asymptotic behavior of $u_\varepsilon$}
Using the regularity results in the previous section we can now describe the behavior of $u_\varepsilon$ and $\rho_\varepsilon$ as $\varepsilon\rightarrow 0$ and prove Theorem \ref{Prin_Theo_ueps}.\\\\
\textbf{Step 1 (Convergence of $u_\varepsilon$ and $\rho_\varepsilon$)} According to section 3.1, $\{u_\varepsilon\}$ is locally uniformly bounded and  equicontinuous, so by the Arzela-Ascoli Theorem after extraction of a subsequence,  $u_\varepsilon(t,x)$ converges locally uniformly to a continuous function $u(t,x)$. Moreover from \eqref{Equi_u}, we obtain that $u$ does not depend on $t$, i.e $u(t,x)=u(x)$.\\
Moreover, from the uniform bounds on $\rho_\varepsilon$ we obtain using \eqref{bornDerRho} that $|\frac{d\rho_\varepsilon}{dt}|$ is bounded too. Thus applying again the Arzela-Ascoli Theorem we can assure that $\rho_\varepsilon(t)$ converges, along subsequences, to a function $\rho(t)$ as $\varepsilon\rightarrow 0$.\\\\
\textbf{Step 2 ($\max_{x\in\R^d}u(x)=0$)} 
Assume that for some $x_0$ we have $0 < \alpha \leq u(x_0)$. Since $u$ is continuous, we have $u(y)\geq\frac{\alpha}{2}$ on $B(x_0,r)$ for some $r > 0$. Thus, using the convergence of $u_\varepsilon$ there exists $\varepsilon_0$ such that for all $\varepsilon\leq\varepsilon_0$ we have $u_\varepsilon(y)\geq \frac{\alpha}{2}$ on $B(x_0,r)$, which implies that  
$$
|B(x_0,r)|\exp\left(\frac{\alpha}{2\varepsilon}\right)\leq\int_{B(x_0,r)}\exp\left(\frac{u_\varepsilon}{\varepsilon}\right)dx\leq \int_{\R^d}n_\varepsilon(t,x)dx=\rho_\varepsilon(t).
$$
Therefore
$\rho_\varepsilon\rightarrow\infty$ as $\varepsilon\rightarrow 0$. This is in contradiction with \eqref{Borns_rho_eps}. Thus $u(x)$ cannot be strictly greater than zero.\\
Next we have thanks to \eqref{Up_Bound_u} that
$$
\lim_{\varepsilon\rightarrow0}\int_{|x|>R}n_\varepsilon(t,x)dx\leq \lim_{\varepsilon\rightarrow0}\int_{|x|>R}e^{\frac{C_1'-C_2'|x|}{\varepsilon}}=0,
$$
for $R$ large enough.\\
From this and Lemma \eqref{Borns_rho_Lemma} we deduce that
\begin{equation}
\label{Lim_n_eps}\rho_m\leq \lim_{\varepsilon\rightarrow0}\int_{|x|\leq R}n_\varepsilon(t,x)dx.
\end{equation}
If $u(x) < 0$ for all $|x| < R$, then $u(x)<-\beta$ for some $\beta>0$, since we know that $u_\varepsilon$ converges locally uniformly to $u$, then there exists $\varepsilon_0$ small enough such that for all $\varepsilon\leq\varepsilon_0$, we have $u_\varepsilon(t,x)<-\frac{\beta}{2}$, $\forall|x|<R$. Therefore
$$
\int_{|x|\leq R}e^{\frac{u_\varepsilon(t,x)}{\varepsilon}}dx\leq\int_{|x|\leq R}e^{-\frac{\beta}{2\varepsilon}}dx=|B(x_0,R)|e^{-\frac{\beta}{2\varepsilon}}\longrightarrow0\quad\mathrm{as}\quad\varepsilon\rightarrow0.
$$ 
Note that this is in contradiction with \eqref{Lim_n_eps}. It follows that $\max_{x\in\mathbb{R}}u(x)=0$.\\\\
\textbf{Step 3 (The equation on $u$)}
We claim that $u(x)=\dis\lim_{\varepsilon\rightarrow0}u_\varepsilon(t,x)$ is a viscosity solution of problem \eqref{LimitEq}. Here, we prove that $u$ is a viscosity subsolution of \eqref{LimitEq}. One can prove, following similar arguments, that $u$ is also a viscosity supersolution of \eqref{LimitEq}.\\
Let us define the auxiliary ``cell problem"
\begin{equation}
\label{CellPb}
\left\{\begin{array}{cr}
\partial_t v=a(t,x)-\rho(t)-\overline{a}(x)+\overline{\rho},&(t,x)\in[0,+\infty)\times\R^d,\\ 
v(0,x)=0,\\
v:\;T-periodic.
\end{array}
\right.
\end{equation}
This equation has a unique smooth solution, that we can explicitly write
$$
v(t,x)=-t(\overline{a}(x)-\overline{\rho})+\int_0^t(a(t,x)-\rho(t))dt.
$$
\\
Let $\phi\in C^\infty$ be a test function and assume that $u-\phi$ has a strict local maximum at some point $x_0\in\R^d$, with $u(x_0)=\phi(x_0)$. We must prove that
\begin{equation}
\label{CondViscSub}
-|\nabla \phi|^2(x_0)-\overline{a}(x_0)+\overline{\rho}\leq 0.
\end{equation}
We define the perturbed test function $\psi_\varepsilon(t,x)=\phi(x)+\varepsilon v(t,x)$, such that $u_\varepsilon-\psi_\varepsilon$ attains a local maximum at some point $(t_\varepsilon,x_\varepsilon)$. We note that $\psi_\varepsilon$ converges to $\phi$ as $\varepsilon\rightarrow0$ since $v$ is locally bounded by definition, and hence one can choose $x_\varepsilon$ such that $x_\varepsilon\to x_0$ as $\varepsilon\to 0$, (see Lemma 2.2 in \cite{GuyBarles}). Then $\psi_\varepsilon$ satisfies
$$
\frac{1}{\varepsilon}\partial_t \psi_\varepsilon(t_\varepsilon,x_\varepsilon)-\varepsilon\Delta \psi_\varepsilon(t_\varepsilon,x_\varepsilon)-\big|\nabla \psi_\varepsilon(t_\varepsilon,x_\varepsilon)\big|^2-a(t_\varepsilon,x_\varepsilon)+\rho_\varepsilon(t_\varepsilon)\leq 0,
$$
since $u_\varepsilon$ is a viscosity solution of \eqref{Pb_u_Epsilon}. The above line implies that
$$
\partial_t v(t_\varepsilon,x_\varepsilon)-\varepsilon\Delta \phi(x_\varepsilon)-\varepsilon^2\Delta v(t_\varepsilon,x_\varepsilon)-\big|\nabla \phi(x_\varepsilon)+\varepsilon\nabla v(t_\varepsilon,x_\varepsilon)\big|^2-a(t_\varepsilon,x_\varepsilon)+\rho_\varepsilon(t_\varepsilon)\leq 0.
$$
Using \eqref{CellPb}, this last equation becomes
\begin{equation}
\label{EqLim}
-\varepsilon\Delta \phi(x_\varepsilon)-\varepsilon^2\Delta v(t_\varepsilon,x_\varepsilon)-\big|\nabla \phi(x_\varepsilon)+\varepsilon\nabla v(t_\varepsilon,x_\varepsilon)\big|^2+(\rho_\varepsilon-\rho)(t_\varepsilon)-\overline{a}(x_\varepsilon)+\overline{\rho}\leq 0.
\end{equation}
We can now pass to the limit as $\varepsilon\rightarrow0$. We know from step 1 that $\rho_\varepsilon\rightarrow\rho$ as $\varepsilon\rightarrow0$. Moreover $v$ is smooth with respect to $x$, with locally bounded derivatives with respect to $x$, thanks to its definition. Using these arguments and letting $\varepsilon\rightarrow0$ in \eqref{EqLim} we obtain \eqref{CondViscSub} which implies that $u$ is a viscosity sub-solution of \eqref{LimitEq}.\\\\
\textbf{Step 4 (Uniqueness and regularity of $u$)}
We first note that, for the case $x\in\R$, that is $d=1$, the solution given by \eqref{Exp_Sol_u} solves \eqref{LimitEq}.\\
In general, Hamilton-Jacobi equations of type \eqref{LimitEq}, may admit more than one viscosity solution. In this case the uniqueness is guaranteed thanks to Proposition 5.4 of Chapter 5 in \cite{Lions}, which assures that, since the RHS of the first equation in \eqref{LimitEq} is null at just one point $(x=x_m)$, and the value of $u$ in this point is known ($u=0$), together with the fact that $\max_{x\in\R^d} u(x)\leq 0$ the solution of \eqref{LimitEq} is unique and is given by 
$$
\begin{array}{l}
u(z)=\sup\Big\{u(x_m)-\dis\int_0^{T_0}\sqrt{\overline{\rho}-\overline{a}(\xi(s))}ds;\ (T_0,\xi)\text{ such that }\xi(0)=x_m,\ \xi(T_0)=z,\ \left|\dfrac{d\xi}{ds}\right|\leq 1,\text{ a.e in }[0,T_0],\\
\qquad\qquad\quad\text{ with }\xi(t)\in\R^d,\forall t\in[0,T_0]\Big\}.
\end{array}
$$
Moreover, in the case $x\in\R$, one can verify that the above formula is equivalent with \eqref{Exp_Sol_u} and such solution $u$ is $C^{3}(\R)$ since $\overline{a}(x)\in C^3(\R)$.\\\\
\textbf{Step 5 (Convergence of $n_\varepsilon$)}
We deal in this step with the result for the convergence of $n_\varepsilon$. To this end we proceed as in Section \ref{ConvDirac}.\\
Call $f_\varepsilon(t,x)=\dfrac{n_\varepsilon(t,x)}{\rho_\varepsilon(t)}$, then $f_\varepsilon$ is uniformly bounded in $L^\infty(\R^+,L^1(\R^d))$. Next, we fix $t\geq0$, and we follow the arguments of Section \ref{ConvDirac} to prove that $f_\varepsilon(t,\cdot)$, as function of $x$, converges, along subsequences, to a measure, as follows
$$
f_\varepsilon(t,\cdot)\rightharpoonup\delta(\cdot-x_m)\quad\mathrm{as}\quad \varepsilon\rightarrow0,
$$
weakly in the sense of measures in $x$.\\
Indeed, from \eqref{Exp_Sol_u} we deduce that 
$$
\max_{x\in\R^d}u(x)=u(x_m)=0.
$$
Then note $\mathcal{O}=\R^d\setminus B_\zeta(x_m)$, for $\zeta$ small enough  and $\psi\in C_c(\mathcal{O})$, such that $\mathrm{supp}\; \psi\subset \mathcal{K}$, for a compact set $\mathcal{K}$
$$
\left|\int_{\mathcal{O}}f_\varepsilon(t,x)\psi(x)dx\right|\leq \frac{1}{\rho_m}\int_{\mathcal{O}}e^{\frac{u_\varepsilon(t,x)}{\varepsilon}}|\psi(x)|dx\leq \frac{1}{\rho_m}\int_{\mathcal{K}}e^{{\frac{u_\varepsilon(t,x)}{\varepsilon}}}|\psi(x)|dx.
$$
From the locally uniform convergence of $u_\varepsilon$, since $ u(x)\leq-\beta$, $\forall x\in\mathcal{K}$, we obtain that there exists $\varepsilon_0>0$ such that $\forall\varepsilon<\varepsilon_0$,  $u_\varepsilon(t,x)\leq-\frac{\beta}{2}$, $\forall x\in\mathcal{K}$, and hence
$$
\int_{\mathcal{K}}e^{\frac{u_\varepsilon(t,x)}{\varepsilon}}|\psi(x)| dx\leq\int_{\mathcal{K}}e^{-\frac{\beta}{2\varepsilon}}|\psi(x)|dx\rightarrow0\quad\mathrm{as}\;\varepsilon\rightarrow0,
$$
since $\psi$ is bounded in $\mathcal{K}.$ Therefore, thanks to the $L^1$ bound of $f_\varepsilon$, we obtain that $f_\varepsilon$ converges weakly in the sense of measures and along subsequences to $\omega\delta(x-x_m)$ as $\varepsilon$ vanishes. We can proceed as in section \ref{ConvDirac} to prove that the convergence is in fact to $\delta(x-x_m)$.\\
Therefore using the convergence result for $\rho_\varepsilon$ we deduce finally \eqref{ConvTheo_eps}.\\\\
\textbf{Step 6 (Identification of the limit of $\rho_\varepsilon$)}
We try now to identify $\rho$ from the explicit expression for $\rho_\varepsilon$. To this end we need to compute the limit of $Q_\varepsilon$. Let $p_\varepsilon$ be the periodic solution of \eqref{p_eps}, we know that $p_\varepsilon(t,x)=\dfrac{n_\varepsilon(t,x)}{\rho_\varepsilon(t)}\dis\int_{\R^d}p_\varepsilon(t,y)dy$. Substituting in $Q_\varepsilon$ we obtain
$$
Q_\varepsilon(t)=\dfrac{\dis\int_{\R^d}a(t,x)p_\varepsilon(t,x)dx}{\dis\int_{\R^d}p_\varepsilon(t,x)dx}=\dfrac{\dis\int_{\R^d}a(t,x)\dfrac{n_\varepsilon(t,x)}{\rho_\varepsilon(t)}\int_{\R^d}p_\varepsilon(t,y)dy dx}{\dis\int_{\R^d}p_\varepsilon(t,x)dx}=\dfrac{\dis\int_{\R^d}a(t,x)n_\varepsilon(t,x)dx}{\rho_\varepsilon(t)}.
$$
From \eqref{ConvTheo_eps} and \eqref{x_m} we deduce that
$$
\lim_{\varepsilon\rightarrow0}Q_\varepsilon(t)=\lim_{\varepsilon\rightarrow0}\int_{\R^d}f_\varepsilon(t,x)a(t,x)dx=a(t,x_m).
$$
Finally we can pass to the limit in the expression \eqref{Exp_rho_eps} for $\rho_\varepsilon$, to obtain \eqref{Exp_rho}, which is in fact the unique periodic solution of the equation \eqref{SysRhoPer}.

\section{Approximation of the moments}
\label{Moments}
In this section we estimate the moments of the population's distribution with a small error, in the case $x\in\R$. To this end, we will use the formal arguments given in Section 1.\\
Using \eqref{Exp_Sol_u}, one can compute a Taylor expansion of order 4 around the point of maximum $x_m$
\begin{equation}
\label{ExpTay_u}
u(x)=-\frac{A}{2}(x-x_m)^2+B(x-x_m)^3+C(x-x_m)^4+O(x-x_m)^5.
\end{equation}
Note also that one can obtain $v$ formally from \eqref{Eq_v_Lap_u} and compute the following expansions
$$
v(t,x)=v(t,x_m)+D(t)(x-x_m)+E(t)(x-x_m)^2+O(x-x_m)^3,\quad w(t,x)=F(t)+O(x-x_m)^2.
$$
The above approximations of $u$, $v$ and $w$ around the maximum point of $u$ allow us to estimate the moments
of the population's distribution with an error of at most order $O(\varepsilon^2)$ as $\varepsilon\rightarrow 0$. \\
Replacing $u_\varepsilon$ by the approximation \eqref{App_u} and using the Taylor expansions of $u$, $v$ and $w$ given above, we can compute
$$
\begin{array}{lcl}
\dis\int_{\R^d}(x-x_m)^kn_\varepsilon(t,x)dx&=\dfrac{e^{v(t,x_m)}\varepsilon^{\frac{k}{2}}}{\sqrt{2\pi}}&\dis\int_{\R^d}y^k e^{\frac{-Ay^2}{2}}\big[1+\sqrt{\varepsilon}\left(By^3+D(t)y\right)\\&&+\varepsilon\left(Cy^4+E(t)y^2+F(t)+\frac{1}{2}(By^3+D(t)y)^2\right)+o(\varepsilon)\big]dy.
\end{array}
$$
Note that, to compute the above integral, we performed a change of variable $x-x_m=\sqrt{\varepsilon}\ y$. Therefore each term $x-x_m$ can be considered as of order $\sqrt{\varepsilon}$ in the integration. Note also that we have used the approximation
$$
n_\varepsilon(t,x)=\frac{1}{\sqrt{2\pi\varepsilon}}e^{\frac{u(x)}{\varepsilon}+v(t,x)+\varepsilon w(t,x)}.
$$ 
The above computation leads in particular to the following approximations of the population size, the phenotypical mean and the variance
$$
\left\{
\begin{array}{l}
\rho_\varepsilon=\dis\int_{\R^d}n_\varepsilon(t,x) dx=\frac{e^{v(t,x_m)}}{\sqrt{A}}\left[1+\varepsilon \left(\frac{15B^2}{2A^3}+\frac{3(C+BD(t))}{A^2}+\frac{E(t)+0,5D(t)^2}{A}+F(t)\right)\right] +O(\varepsilon^2),\\
\mu_\varepsilon=\dis\frac{1}{\rho_\varepsilon(t)}\int_{\R^d}x\ n_\varepsilon(t,x)dx= x_m+\varepsilon\left(\frac{3B}{A^2}+\frac{D(t)}{A}\right)+O(\varepsilon^2),\\
\sigma^2_\varepsilon=\dis\frac{1}{\rho_\varepsilon(t)}\int_{\R^d}(x-\mu_\varepsilon)^2n_\varepsilon(t,x)dx=\frac{\varepsilon}{A}+O(\varepsilon^2).
\end{array}
\right.
$$

\section{Some biological examples}
\label{examples_bio}
In this section, we present two examples where two different growth rates $a(t,x)$ are considered. In particular, the fluctuations act on different terms in the two examples, (they act respectively on the optimal trait and on the pressure of the selection).\\
We are motivated by a biological experiment in  \cite{Ketola}, where a population of bacterial pathogen \textit{Serratia marcescens} was studied. In this experiment several populations of Serratia marcescens were kept in environments with constant or fluctuating temperature for several weeks. Then, their growth rates were measured in different environments. In particular, 
it was observed that a population of bacteria that evolved in periodically  fluctuating temperature (daily variation between $24^\circ$C and $38^\circ$C, mean $31^\circ$C) outperforms the strains that evolved in cons\-tant temperature 
($31^\circ$C), when both strains are allowed to
compete in a constant environment with temperature $31^\circ$C.  Note that this is a surprising effect, since one expects that the population evolved in a constant environment would select for the best trait in such environment.\\
Here, we estimate the moments of the population's distribution and the mean effective fitness of the population in a constant enviroment for our two examples. We will observe that the second example, where the fluctuations act on the pressure of selection, allows to capture the phenomenon observed in the experiment in \cite{Ketola}. Our analysis shows that, in presence of the mutations and while the fluctuations act on the pressure of the selection, a fluctuating environment can select for a population with the same mean phenotypic trait but with smaller variance and in this way lead to more performant populations. \\
Note that our analysis is very well adapted to the mentioned experiment, since we first study the long time behavior of the phenotypical distribution and we find that it is the periodic solution of a selection-mutation equation. This distribution  corresponds to the phenotypical distribution of a population evolved in a periodic environment. Next, we characterize such distribution assuming small mutations. We remark that, although our analysis provides a possible explanation for the observed experience in \cite{Ketola}, one should  go back to the biological experiment and compare the population's distribution with our results to test this interpretation. 
 
\subsection{Oscillations on the optimal trait}
In this subsection we study a case where the optimal trait fluctuates periodically. We show that in this case,  the population's phenotypical mean follows the optimal trait with a delay.\\
We choose, as periodic growth rate
$$
a(t,x)=r-g(x-c\sin b t)^2, 
$$ 
where  $r,g,c$ and $b$ are positive constants. Here $r$ represents the maximal growth rate, $g$ models the selection pressure and the term $c\sin b t$ models the oscillations of the optimal trait with period $\frac{2\pi}{b}$ and amplitude $c$. \\
We compute the mean of $a(t,x)$ 
$$
\overline{a}(x)=\frac{b}{2\pi}\int_0^{\frac{2\pi}{b}}a(t,x)dt=r-g\left(x^2+\frac{c^2}{2}\right),
$$
and we observe that the maximum of $\overline{a}(x)$ is attained at $x_m=0$.\\
From here we can also compute the mean population size $\overline{\rho}_\varepsilon$. We do not provide an explicit formula for $\rho_\varepsilon$, but only for its mean value, in order to keep the simpler expression, however, for the higher order moments we give the exact value until order 1 in $\varepsilon$.\\
Let $u(x)$ be given by \eqref{Exp_Sol_u}, which can be rewritten in this specific example as follows
$$
u(x)=-\left\vert\int_0^x\sqrt{gy^2}dy\right\vert=-\frac{\sqrt{g}}{2}x^2,
$$
then we obtain, from \eqref{LimitEq} and the second equation in \eqref{Eq_v_Lap_u}
$$
\overline{\varrho}=\overline{a}(0)=r-\dfrac{g c^2}{2},\qquad \overline{k}=\Delta u(0)=-\sqrt{g}.
$$
On the other hand, by substituting $u(x)$ in \eqref{ExpTay_u} we find $A=\sqrt{g}$, $B=C=0$, and also by substitution in \eqref{Eq_v_Lap_u} we obtain
$$
\int_0^{\frac{2\pi}{b}}\partial_x v(t,x)dt=0,\quad \partial_t\partial_x v=2gc\sin b t.
$$
We deduce that 
$$
\partial_x v(0,x)=\dfrac{-2cg}{b}\quad\mathrm{and}\quad \partial_x v(t,x)=-\dfrac{2cg}{b}\cos b t.
$$
Now we are able to compute the approximations of order one with respect to $\varepsilon$ of the population mean size $\overline{\rho}_\varepsilon$, the phenotypical mean $\mu_p$ and the variance $\sigma^2_p$ of the population's distribution, following the computations we have done in the previous section, that is
$$
\mu_p(t)\approx\dfrac{2\varepsilon c}{b}\sqrt{g}\sin\left(b t-\frac{\pi}{2}\right),\quad \sigma^2_p\approx\frac{\varepsilon}{\sqrt{g}},\quad \overline{\rho}_\varepsilon\approx r-\dfrac{g c^2}{2}-\varepsilon\sqrt{g}.
$$
We observe, in  fact, that the mean trait $\mu_p(t)$ oscillates with the same period as the optimal trait with a delay $\frac{\pi}{2b}$, and a small amplitude, as was suggested for instance in \cite{Lande}.\\
We also compute $\widetilde{F}_p(\tau)$ the mean fitness of the population (evolved in the periodic environment), in an environment with temperature $\tau$, and hence with growth rate $a(\tau,x)$
\begin{equation}
\label{mean-fitness}
\widetilde{F}_p(\tau)=\dis\int_{\R^d}a(\tau,x)\frac{1}{T}\int_0^T\frac{n_\varepsilon(t,x)}{\rho_\varepsilon(t)}dt dx,
\end{equation}
which can be approximated for this example at $\tau=\frac{\pi}{b}$ by
$$
\quad\widetilde{F}_p(\pi/b)\approx r-\varepsilon\sqrt{g}.
$$
Note that \eqref{mean-fitness} is the quantity which has been measured in the experiment in \cite{Ketola}.\\
\\\\
We next consider a population which has evolved in a constant environment with $t=\frac \pi b$, (mean time), that is when the growth rate is given by  $a(\pi/b,x)=r-gx^2$. With such constant in time growth rate, the density of the population's distribution  converges in long time to the unique solution of the following stationary equation
$$
\begin{cases}
-\varepsilon^2\partial_{xx} n_c=n_c \big( r-gx^2-\rho_c\big),\\
\rho_c=\int_\R n_c dx.
\end{cases}
$$
The solution of the above equation can be computed explicitly and is given by
$$
n_c=\rho_c \frac{g^{\frac 14}}{\sqrt{2\pi \varepsilon}} \exp \left(\frac{-\sqrt{g}x^2}{2\varepsilon}\right),\quad \rho_c=r-\varepsilon\sqrt{g}.
$$
We can then compute the population mean size $\overline{\rho}_c$, the mean trait $\mu_c$ and the variance $\sigma_c^2$ for such population
$$
\mu_c=0,\quad \sigma_c^2=\frac{\varepsilon}{\sqrt{g}},\quad \overline{\rho}_c=r-\varepsilon\sqrt{g}.
$$
Here we observe that the size of a population evolved in a constant environment $\overline{\rho}_c$ is greater than the mean size of a population evolved in a fluctuating environment $\overline{\rho}_\varepsilon$.\\
Moreover, the mean fitness of such population, in an environment with the same temperature ($t=\frac{\pi}{b}$), can be computed as below
$$
\widetilde{F}_c=\dis\int_{\R}a(\pi/b,x)\frac{n_c(x)}{\rho_c} dx=r-\varepsilon\sqrt{g}.
$$
We hence obtain that, independently of the choice of constants $r,g$ and $c$, both populations (the one evolved in a constant environment and the other evolved in a fluctuating environment) have the same mean fitness at the same constant environment, up to order $\varepsilon$. This result does not correspond to what was observed in the experiment of \cite{Ketola}. In the next subsection we consider another example where the oscillations act differently on the growth rate and where the outcome corresponds more to the experiment of \cite{Ketola}.

\subsection{Oscillations on the pressure of the selection}
In this subsection, we study an example where the fluctuations act on the pressure of the selection. We show that in this case a population evolved in a fluctuating environment (for instance  with fluctuating temperature), may outperform a population evolved in a constant environment, in such constant environment. \\

\noindent
Here, we consider the following periodic growth rate
$$
a(t,x)=r-g(t)x^2,
$$
where $r>0$ is the maximal growth rate as in the previous example in Section 6.1 and $g(t)$ is a $1-$periodic function which models the oscillating pressure of selection.\\
Then $\overline{a}$ is given by
$$
\overline{a}(x)=r-\overline{g}x^2 \quad\mathrm{with}\quad\overline{g}=\int_0^1g(t)dt
,$$
where again the maximum of $\overline{a}(x)$ is attained at $x_m=0$.\\
We compute $u$ and $\partial_x v$ as before and obtain
$$
u(x)=-\frac{\sqrt{\overline{g}}}{2}x^2,\quad \partial_x v(t,x)=2x\left(t\overline{g}-\int_0^tg(t')dt'+\int_0^1\int_0^tg(t')dt'dt-\frac{\overline{g}}{2}\right).
$$
We compute again $\overline{\varrho}$ and $\overline{k}$, from \eqref{LimitEq} and \eqref{Eq_v_Lap_u}, in order to approximate $\overline{\rho}_\varepsilon$,  that is
$$
\overline{\varrho}=\overline{a}(0)=r,\quad \overline{k}=\Delta u(0)=-\sqrt{\overline{g}}.
$$
Then from the expression of $u$, again with the help of the formula from the previous section we obtain $A=\sqrt{\overline{g}}$, $B=C=0$ and $D(t)=0$.\\
We next compute the approximations of order one with respect to $\varepsilon$ of the population mean size $\overline{\rho}_\varepsilon$, the phenotypical mean $\mu_p$ and the variance $\sigma_p^2$ of the population's distribution which are given by 
$$
\mu_p\approx0,\quad \sigma_p^2\approx\frac{\varepsilon}{\sqrt{\overline{g}}},\quad \overline{\rho}_\varepsilon\approx r-\varepsilon \sqrt{\overline{g}}.
$$
Analogously to the previous example, we also compute $\widetilde{F}_p(\tau)$ the mean fitness of the population (evolved in the periodic environment), in an environment with temperature $\tau=\frac{1}{2}$, and hence with growth rate $a(\frac{1}{2},x)$, which can be approximated for this example as
$$
\quad\widetilde{F}_p(1/2)\approx r-\varepsilon\frac{g(1/2)}{\sqrt{\overline{g}}}.
$$

\noindent
We next consider, a population which has evolved in a constant environment with $t=\frac 1 2$, that is when the growth rate is given by  $a(1/2,x)=r-g(1/2)x^2$. Again, the density of the population's distribution converges in long time to the unique solution of the following stationary solution
$$
\begin{cases}
-\varepsilon^2\partial_{xx} n_c=n_c \big( r-g(1/2)x^2-\rho_c\big),\\
\rho_c=\int_\R n_c dx.
\end{cases}
$$
The explicit solution of the above equation is given by
$$
n_c=\rho_c \frac{g(1/2)^{\frac 14}}{\sqrt{2\pi \varepsilon}} \exp \big( \frac{-\sqrt{g(1/2)}x^2}{2\varepsilon } \big),\quad \rho_c=r-\varepsilon\sqrt{g(1/2)},
$$
from where we obtain the following population mean size $\overline{\rho}_c$, mean $\mu_c$ and variance $\sigma_c^2$ for such population
$$
\mu_c=0,\quad \sigma_c^2=\frac{\varepsilon}{\sqrt{g(1/2)}},\quad \overline{\rho}_c=r-\varepsilon\sqrt{g(1/2)}.
$$
Moreover, the mean fitness of such population, in an environment with the same temperature ($t=1/2$), can be computed as below
$$
\widetilde{F}_c=\dis\int_{\R^d}a(1/2,x)\frac{n_c(x)}{\rho_c} dx=r-\varepsilon\sqrt{g(1/2)}.
$$
We remark that if we choose $g(t)$ such that
\begin{equation}
\label{choice_g}
{\int_0^1g(t)dt}> {g\left(1/2 \right)},
\end{equation}
then we have
$$
\overline{\rho}_\varepsilon<\overline{\rho}_c, \quad\sigma^2_p<\sigma^2_c\quad\mathrm{and}\quad \widetilde{F}_c<\widetilde{F}_p(1/2).
$$
Here we observe that, for this choice of $g$ satisfying \eqref{choice_g}, the population evolved in a periodic environment has a larger fitness, in an environment with constant temperature ($\tau=1/2$) than the one evolved in a constant temperature ($\tau=1/2$). This property corresponds indeed to  what was observed in the biological experiment in \cite{Ketola}. Note that both of these environments select for populations with the same phenotypic mean trait $x=0$. However, the population evolved in a periodic environment has a smaller variance comparing to the one evolved in a constant environment. This makes the population evolved in the periodic environment more performant. This example shows that the phenomenon observed in the experiment of \cite{Ketola} can also be observed in mathematical models.

\section*{Acknowledgments}
Both authors thank Sylvain Gandon for fruitful discussions on the biological motivations. The first author is also immensely thankful to the ``Fondation de Sciences Math\'ematiques de Paris" (FSMP) for the opportunity of a second year of Master which helped very much to her formation. The second author is also grateful for partial funding from the European Research Council (ERC)
under the European Union's Horizon 2020 research and innovation programme (grant agreement No
639638), held by Vincent Calvez, and from the french ANR projects KIBORD ANR-13-BS01-0004
and MODEVOL ANR-13-JS01-0009.


\end{document}